\documentclass{article}
\usepackage{geometry}
\usepackage{graphicx} 
\usepackage{tikz}
\usepackage{tikz-cd}
\usepackage{lmodern}

\usepackage{url,amsfonts,amsthm,microtype,mathtools,bbm,stmaryrd,enumerate,amssymb,verbatim}

\usepackage{graphicx}

\usepackage{mathrsfs}

\usepackage[pagebackref,colorlinks,citecolor=Mahogany,linkcolor=Mahogany,urlcolor=Mahogany,filecolor=Mahogany]{hyperref}
\usepackage{setspace}
\usepackage{cancel}

\usepackage{tikz,tikz-cd,hf-tikz}
\usetikzlibrary{matrix,calc,positioning,arrows,decorations.pathreplacing,patterns,knots,tikzmark}

\usepackage{wasysym} 

\usepackage{float} 

\newcommand{\cB}{\mathcal{B}}

\newcommand{\Z}{\mathbb{Z}}
\newcommand{\R}{\mathbb{R}}

\newdimen\R
\definecolor{lavenderblue}{rgb}{0.8, 0.8, 1.0}
\definecolor{lightmauve}{rgb}{0.86, 0.82, 1.0}
\definecolor{brightlavender}{rgb}{0.75, 0.58, 0.89}

\usepackage{wrapfig,caption}
\newtheorem{theorem}{Theorem}[section]
\newtheorem{lemma}[theorem]{Lemma}

\newtheorem{proposition}[theorem]{Proposition}
\newtheorem{corollary}[theorem]{Corollary}
\newtheorem*{corollary*}{Corollary}

\theoremstyle{definition}

\newtheorem{definition}[theorem]{Definition}

\theoremstyle{remark}
\newtheorem{example}[theorem]{Example}
\newtheorem{remark}[theorem]{Remark}

\usepackage{mdframed}

\newmdtheoremenv{boxedaside}{Aside}

\usepackage{xargs}                      
\definecolor{todocolor}{HTML}{454545}
\usepackage[colorinlistoftodos,prependcaption,textsize=tiny,linecolor=todocolor,backgroundcolor=todocolor!25,bordercolor=todocolor]{todonotes}

\definecolor{lav}{HTML}{BBBBEE}
\definecolor{lav}{HTML}{9393ed}
\newcommandx{\margincomment}[2][1=]{\todo[linecolor=lav,backgroundcolor=lav!25,bordercolor=lav,#1]{#2}}

\newcommand{\tikzs}{\begin{center}\begin{tikzcd}}
\newcommand{\tikze}{\end{tikzcd}\end{center}}

\newcommand{\Hom}{\text{Hom}}

\usetikzlibrary{calc}

\def\sqtwoL (#1,#2){
  \draw (#1,#2) .. controls (#1-1,#2+1) .. (#1,#2+2);
}

\def\sqtwoR (#1,#2){
  \draw (#1,#2) .. controls (#1+1,#2+1) .. (#1,#2+2);
}

\def \sqtwoCR (#1,#2){
   \draw (#1,#2) .. controls (#1+0.75,#2+.5) and (#1+1.125,#2+2) .. (#1+1.5,#2+2);
}

\def \sqtwoCL (#1,#2){
   \draw (#1,#2) .. controls (#1-0.75,#2+.5) and (#1-1.125,#2+2)  .. (#1-1.5,#2+2);
}

\def \sqone (#1,#2){
  \draw (#1, #2) -- (#1, #2+1);
}

\newcommand{\Po}{\mathbb{P}^1}

\definecolor{Mahogany}{rgb}{0.75,0.25,0}
\usepackage{xcolor}

\usepackage{subcaption}

\newcommand\smvee{\raise0.9ex\hbox{$\scriptscriptstyle\vee$}}
\newcommand{\Aut}{{\rm Aut}}
\newcommand{\Sym}{{\rm Sym}}

\newcommand{\Jac}{{\rm Jac}}
\newcommand{\M}{\mathcal{M}}

\newcommand{\B}{\mathcal{B}}
\newcommand{\Bb}{\overline{\mathcal{B}}}

\newcommand{\OO}{\mathcal{O}}

\newcommand{\mb}{\overline{M}}
\newcommand{\mbr}{\mb^{\red}}
\newcommand{\hti}{\widetilde{{\mathcal{H}}}}
\newcommand{\HH}{\mathcal{H}}
\newcommand{\mt}{\widetilde{M}}
\newcommand{\mtc}{\widetilde{M}^{\circ}}
\newcommand{\D}{\mathcal{D}}
\newcommand{\Du}{\Xi}
\newcommand{\Dr}{\Delta}
\newcommand{\Duct}{\delta}

\newcommand{\CC}{\mathcal{C}}
\newcommand{\LL}{\mathbb{L}}
\newcommand{\Ext}{{\rm Ext}}
\newcommand{\Spec}{{\rm Spec}}

\newcommand{\br}{{\rm br}}

\newcommand{\ub}{{\rm ub}}
\newcommand{\red}{{\rm red}}
\newcommand{\Mb}{\overline{\mathcal{M}}}
\newcommand{\Pic}{\textup{Pic}}

\title{The $p$-rank stratification of the moduli space of double covers of a fixed elliptic curve} 
\author{Kevin Chang, Du\v{s}an Dragutinovi\'{c}, Steven R. Groen, Yuxin Lin, \\ Natalia Pacheco-Tallaj, Deepesh Singhal}
\date{}

\begin{document}

\maketitle
\begin{abstract}
    In this paper we investigate the $p$-rank stratification of the moduli space of curves of genus $g$ that admit a double cover to a fixed elliptic curve $E$ in characteristic $p>2$. We show that the closed $p$-rank strata of this moduli space are equidimensional of the expected dimension. We also show the existence of a smooth double cover of $E$ of all the possible values of the $p$-rank on this moduli space.
\end{abstract}

\section{Introduction}
\label{sec:intro}

Let $k$ be an algebraically closed field of characteristic $p>2$ throughout this paper. We study the $p$-rank stratification of the moduli space of double covers of a fixed elliptic curve over $k$. In characteristic $p$, abelian varieties carry several interesting discrete invariants, such as the $p$-rank or the Newton polygon. The corresponding invariants for curves are defined to be those of their Jacobians. These invariants induce stratifications on the moduli spaces of abelian varieties, and consequently on the moduli spaces of curves.

In this paper, we focus on the $p$-rank of curves. The $p$-rank of an abelian variety $A$ over $k$ is the integer $f_A:=\dim(\text{Hom}(\mu_p, A))$ such that $\#A[p](k)=p^{f_A}$; it satisfies $0 \leq f_A \leq \dim(A)$. The $p$-rank of a curve is defined to be the $p$-rank of its Jacobian. By \cite{oda}, the $p$-rank of a curve $C$ over $k$ also equals the semisimple rank of the Frobenius operator on $H^1(C, \mathcal{O}_C)$. The $p$-rank is invariant under isogenies, and the Newton stratification refines the $p$-rank stratification.

The $p$-rank stratifications of various moduli spaces have been studied extensively. Given a moduli space $\mathcal{M}$ of curves or abelian varieties, denote by $V_f (\mathcal{M})$ the \emph{closed $p$-rank $f$ stratum}, consisting of objects whose $p$-rank is at most $f$. It is known (\cite[Lemmas 1.4 and 1.6]{Oort1974} and \cite[Theorem~7.(1)]{Koblitz}) that $V_f (\mathcal{A}_g)$ is pure of codimension $g-f$. It is a subtler question how the $p$-rank stratification intersects the Torelli locus. It is proved in \cite[(2.3)]{faber2004complete} that closed $p$-rank strata of~$\overline{\mathcal{M}}_g$ are pure and have the expected codimension $g-f$. As a consequence, the generic point of any component of $V_f (\overline{\mathcal{M}}_g)$ corresponds to a smooth curve. Similar investigations for certain covers of the projective line followed. In \cite{glass2005hyperelliptic}, it is shown that the same holds for the closed $p$-rank strata of $\overline{\mathcal{H}}_g$, the moduli space of hyperelliptic curves of genus $g$. The $p$-rank stratification of the moduli space of \emph{Artin-Schreier covers}, separable degree $p$ covers of the projective line, is explored in \cite{PriesAS}. The $p$-rank strata of the moduli space of degree $l$ cyclic covers of $\mathbb{P}^1$, where $l\neq p$ is prime, are studied in \cite{Ekin2022cyclic}. In \cite[(6.1)]{faber2004complete}, $p$-ranks of unramified double covers of genus $2$ curves are studied. In \cite{Bouw_2001}, ramified covers of degree coprime to $p$ of generic $r$-pointed curves are studied for sufficiently large $p$. 

A natural next step is to investigate the $p$-ranks of potentially ramified double covers of a fixed elliptic curve $E = (E, O_E)$ over $k$. Note that the $p$-rank $f_E$ of $E$ is either $f_E = 0$ or $f_E = 1$, depending on whether $E$ is supersingular or ordinary.  
Let $\mathcal{B}_{E,g}$ denote the moduli space of smooth genus-$g$ curves over $k$ that admit a $\Z/2\Z$-cover to $E$, and let $\overline{\mathcal{B}}_{E,g}$ denote its Deligne-Mumford compactification. Furthermore, let $\mathcal{B}_g$ denote the bielliptic locus and $\overline{\mathcal{B}}_g$ its Deligne-Mumford compactification. Our main result is that the closed $p$-rank strata of $\overline{\mathcal{B}}_{E,g}$ and $\overline{\mathcal{B}}_g$ are pure and have the expected (co)dimension.

\begin{theorem} \label{thm:mainintro}
Let $g \geq 2$ and $p>2$. For any $f_E \leq f \leq g - 1 + f_E$, the locus $V_f(\Bb_{E, g})$ is pure of dimension $g - 2 + f - f_E$. As a consequence, $V_f(\overline{\mathcal{B}}_g)$ is pure of dimension $g - 2 + f$, for any $0 \leq f \leq g$.
\end{theorem}

Furthermore, our geometric insights lead us to the following corollary.

\begin{corollary} The following statements hold for any $g\geq 2$ and any prime number $p>2$.
\begin{enumerate}
    \item For any $0\leq f \leq g - 2$, there is a component of $V_f(\overline{\mathcal{B}}_{g})$ that consists entirely of singular curves. 
    \item For any $0\leq f \leq g$, there is a smooth bielliptic curve $C$ of genus $g$ in characteristic~$p$ whose $p$-rank equals $f$, except for the case $(p, g, f) = (3, 2, 0)$.
\end{enumerate}
\label{cor:p-rankintro}
\end{corollary} 

The spaces $\mathcal{B}_{E,g}$ and $\mathcal{B}_g$ are of particular interest because they have interesting arithmetic applications. 
Any smooth supersingular curve of genus $g$ admits a map to 
a supersingular elliptic curve. 
By choosing $E$ to be supersingular and considering the space $\mathcal{B}_{E,g}$,
we learn more about the case where such a map has degree $2$.  Understanding this could also lead to a better understanding of whether there exists a smooth supersingular curve of genus $g$. In fact, the existence of a smooth supersingular curve $C$ of genus $g = 4$ in characteristic $p > 3$ (and hence for all $p > 0$) was established in \cite{Kudo2020}, where $C$ was constructed as the normalization of the fiber product of two (well-chosen) supersingular elliptic curves $E_i \to \Po$, for $i = 1, 2$, over $\Po$. Note that the induced maps $C \to E_1$ and $C \to E_2$ are double covers. These moduli spaces are also relevant for Oort's conjecture, which states that the generic point of any component of the locus of supersingular abelian varieties has the minimal automorphism group, that is $\{\pm 1 \}$. An analog of this statement for supersingular curves was proved for $g = 4$ and $p > 2$ in \cite{dragutinović2024oortsconjectureautomorphismssupersingular}, which implies Oort's conjecture for $g = 4$ and $p > 2$. The proof uses geometric insights into the $p$-rank $0$ stratum of $\mathcal{B}_{4}$. 

\subsection{Structure of this paper}

The proof of our main theorem makes extensive use of the boundary of $\mathcal{B}_{E,g}$ in $\overline{\mathcal{M}}_g$, and a large part of the paper is devoted to understanding that boundary. The reason the boundary is so useful is that, while it can be extremely difficult to construct a smooth curve with a given $p$-rank, it is relatively easy to construct singular curves with a given $p$-rank. Namely, one can clutch together multiple well-understood curves of smaller genus at ordinary double points. This provides a point in $V_f(\overline{\mathcal{B}}_{E,g})$. If the boundary components of $\overline{\mathcal{B}}_{E,g}$ are well-understood, then, in many cases, one can argue by dimension reasons that $V_f(\mathcal{B}_{E,g})$ is dense inside $V_f(\overline{\mathcal{B}}_{E,g})$ and, in particular, also non-empty. Moreover, the dimension of the boundary can yield the dimension of the smooth locus.

To this end, in Section~\ref{section:construct_meg_stable}, we introduce the Deligne-Mumford stack $\overline{M}_{E,g;n}$, which parametrizes admissible double covers of a genus~1 curve with one component identified with $E$, along with~$n$ marked points on the source curve. The substack $M_{E,g;n}$ consists of points where the source curve is smooth. This provides the appropriate setting for proving the results we need about~$\overline{\mathcal{B}}_{E,g}$, which is the image of the map $\mb_{E,g} \to \overline{\mathcal{M}}_g$ that extracts the source curve.
In Lemma~\ref{lem: boundary is clutching} and Corollary~\ref{cor: smooth locus is dense}, we explicitly describe the boundary of $\mb_{E,g}$ and deduce that its smooth locus $M_{E,g}$ is dense. Since the forgetful morphism $\overline{M}_{E,g} \to \overline{\mathcal{M}}_g$ is not quasi-finite, we consider a modified moduli space $\mb_{E,g;n}^{\red}$, consisting of covers whose branch points on $E$ sum to the unit element $O_E$. In Proposition~\ref{forgetful map is quasi-finite}, we prove that the forgetful map $\mb_{E,g}^{\red} \to \Mb_g$ is quasi-finite on the smooth locus~$\mbr_{E,g}$ and generically finite on every boundary component.

In Section \ref{Sec: p rank strat MEg}, these results are used to analyze the $p$-rank stratification of $\mathcal{B}_{E,g}$. It is shown that every $p$-rank stratum intersects non-trivially with the boundary $\overline{\mathcal{B}}_{E,g} - \mathcal{B}_{E,g}$. The purity theorem~\cite[Theorem 4.1]{de2000purity} then provides a lower bound on the dimension of $V_f (\overline{\mathcal{B}}_{E,g})$, which we prove to be sharp following the strategy of \cite[(2.3)]{faber2004complete}, resulting in Theorem~\ref{thm:mainintro}. More sophisticated arguments are required to handle the case when $E$ is a supersingular curve.
Finally, we prove Corollary~\ref{cor:p-rankintro} using geometric insights into the $p$-rank $f$ strata and the inductive approach of \cite[Theorem~6.4]{pries_current_results}.

In Section \ref{sec:applications}, applications of Theorem~\ref{thm:mainintro} to the existence of smooth supersingular bielliptic curves are discussed. Finally, in Section~\ref{sec: explicit} we construct explicit smooth double covers of $E$ of low $p$-rank.  
In Appendix~\ref{appendix:geometry_meg}, necessary results about the geometry of relevant Deligne-Mumford stacks such as smoothness and properness are proved.


\section*{Acknowledgment}

This project originated at the Arizona Winter School in 2024, of which we would like to thank the organizers cordially. In particular, we would like to thank Rachel Pries for organizing the lecture series on the Torelli locus, suggesting this project topic and for many helpful conversations and instructions. We would like to thank {George Nicolas Diaz-Wahl, Caleb Ji, Rose Lopez and Isaiah Mindich} for contributions to this project during the Arizona Winter School. We would also like to thank Ben Moonen for his helpful remark. DD is supported by the Mathematical Institute of Utrecht University. NPT is supported by the National Science Foundation under Grant No. DGE-2141064.

It is our pleasure to thank an anonymous referee for helpful comments.

\section{Moduli of double covers of $E$}\label{section:construct_meg_stable}

Let $k$ be an algebraically closed field of characteristic $p>2$, and let $E$ be a fixed elliptic curve over $k$. Our goal is to understand the locus $\B_{E,g}$ of smooth curves of genus $g$ that admit a double cover to $E$. In order to understand the boundary of $\B_{E,g}$ in $\overline{\mathcal{M}}_g$, we instead study the moduli space that parametrizes double covers of $E$ together with the map to $E$, which has a more accessible deformation theory.
More precisely, we construct the Deligne-Mumford stack $\mb_{E,g}$ of genus-$g$ double covers of a fixed elliptic curve $E$. Details about the construction and geometry of~$\mb_{E,g}$ can be found in Appendix~\ref{appendix:geometry_meg}. By defining $\mb_{E,g}$ as a substack of stable maps to a suitable Deligne-Mumford stack, we obtain its deformation theory and dimension. This provides insight into the boundary components of $\mb_{E,g}$. Note that $\Bb_{E,g}$, the closure of $\B_{E,g}$ in $\overline{\mathcal{M}}_g$, equals the image of the forgetful map $\overline{M}_{E,g} \to \overline{\mathcal{M}}_g$. Lemma~\ref{lem: boundary is clutching} describes the boundary of $M_{E,g}$, and Proposition~\ref{forgetful map is quasi-finite} allows us to use that description to understand the boundary of $\B_{E,g}$.

We will work with the following stacks, describe their $k$-points, and give a rigorous definition of the stacks using twisted stable maps in Definition~\ref{def: mbEgn}.

\begin{itemize}
\item $\mb_{E,g;n}$ is the stack parametrizing the following data: \begin{itemize}
\item A genus-$1$ nodal curve $C$ equipped with a map $C\to E$, with one component identified with $E$ by the map and all other components copies of $\mathbb{P}^1$. %

\item An admissible $\mathbb{Z}/2\mathbb{Z}$-cover $D \to C$ in the sense of \cite[Definition 4.3.1]{abramovich2003twisted} with $D$ a nodal curve of genus $g$.

\item A reduced degree $n$ divisor $\Sigma_{\ub} \subset C$ above which $D \to C$ is unramified.
\end{itemize}

We will write $\mb_{E,g} \coloneqq \mb_{E,g;0}$. For the unfamiliar reader, an admissible $\mathbb{Z}/2\mathbb{Z}$-cover $D \to C$ is essentially a double cover by a nodal curve $D$ in which nodes map to nodes, and ramification over a node is allowed provided that the ramification index at the node is the same when restricting the cover to each component containing it. 
Within $\mb_{E,g;n}$, we have the open substack $M_{E,g;n}$ where $C \cong E$. We think of $M_{E,g}$ as the moduli space of smooth genus~$g$ double covers of $E$, and of $M_{E,g;n}$ as the moduli space of genus~$g$ double covers of $E$ together with an unbranched divisor on $E$.

We will denote points of $\mb_{E,g;n}$ by pairs $(D \to C, \Sigma_{\ub})$, and points of $\mb_{E,g}$ by $D \to C$. Sometimes, we will write $D \to C \to E$ to emphasize the data of the map $C \to E$, but for the most part, we will just think of $C$ as having a single component identified with $E$.
For a $k$-point~$(D \xrightarrow[]{\pi} C, \Sigma_{\ub})$ of $\mb_{E,g;n}$, we will use $\Sigma_{\br} \subset C$ to denote the degree $2g - 2$ divisor of branch points in the smooth locus of $C$. We note that $(C, \Sigma_{\br} \cup \Sigma_{\ub})$ and $(D, \pi^{-1}(\Sigma_{\br} \cup \Sigma_{\ub})^{\red})$ must be stable (see Remark \ref{rem: describing the $k$ point of $M_{E,g,n}$}).

\item $\mt_{E,g;n}$ is the stack parametrizing the following data: \begin{itemize}
\item A $k$-point $(D \to C,\Sigma_{\ub}) \in \mb_{E,g;n}$.

\item An ordering of the branch points and the unbranched points $\Sigma_{\ub}$.
\end{itemize}
Note that there is an \'{e}tale $S_{2g - 2} \times S_n$-cover $\mt_{E,g;n} \to \mb_{E,g;n}$. The factor $S_{2g-2}$ acts on $\Sigma_{\br}$ and the factor $S_n$ acts on $\Sigma_{\ub}$. 
\end{itemize}

We will use the following result, which is proven in Appendix \ref{appendix:geometry_meg} as Proposition \ref{proposition:meg_proper_dm}, Theorem~\ref{theorem:meg_smooth}, and Corollary \ref{corollary:megn_irred}. The proofs and further details of our constructions are deferred to the appendix for readability, since they are not essential for understanding the rest of the paper.

\begin{theorem}\label{theorem:meg_properties_summary}
For $g \ge 2$ and $n \ge 0$, $\mb_{E,g;n}$ is an irreducible smooth proper Deligne-Mumford stack of dimension $2g - 2 + n$.
\end{theorem}

\subsection{The boundary $\mb_{E,g}-M_{E,g}$}\label{section:boundary_description}
Let $\mt_{E,g;n}$ and $\mb_{E,g;n}$ be as defined above; for stability, we assume $g \geq 1$, $n \geq 0$, and $2g + n \geq 3$. 
We also consider the moduli of pointed hyperelliptic curves: let $\hti_{g;n}$ denote the Deligne-Mumford stack of stable $\mathbb{Z}/2\mathbb{Z}$-covers of the projective line with $2g + 2$ ordered smooth branch points and $n$ smooth unbranched marked points on the base curve; for stability, we assume $g \geq 0$, $n \geq 0$, and $2g + n \geq 1$. 
Furthermore, if $\Gamma$ is an irreducible substack parametrizing certain curves, we call the curve representing its generic point the ``generic curve'' of $\Gamma$.

In this section, we describe the boundary strata of $\mb_{E,g}$ in terms of images of clutching morphisms and, in particular, show that $M_{E,g}$ is dense in $\mb_{E,g}$.
We require two types of clutching maps, depending on whether the double cover is branched at the node where the curves are clutched. A visual representation of the clutching maps on the source and target curves is given in Figure~\ref{fig:clutching}.

\begin{figure}[h]
  \centering
    \begin{subfigure}[t]{0.34\textwidth} \label{fig: ur 1}
    \centering
    \scalebox{0.7}{

    \begin{tikzpicture}[yscale=cos(70)]

  \draw[thick, double distance=5mm] (-1.45,0) ++(0:1) arc (0:180:1);
    \draw[thick, double distance=5mm] (-1.45,0) ++(180:1) arc (180:360:1);

  \draw[thick, double distance=5mm] (0.95,0) ++(0:1) arc (0:180:1);
    \draw[thick, double distance=5mm] (0.95,0) ++(180:1) arc (180:360:1);

  \fill[white]
          (-0.75,1.5) .. controls (-0.45,1.2) and (-0.05,1.2) .. (0.25,1.5)
          -- (0.25,-1.5) .. controls (-0.05,-1.2) and (-0.45,-1.2) .. (-0.75,-1.5)
          -- cycle;

  \draw[thick]
          (-0.75,1.5) .. controls (-0.45,1.2) and (-0.05,1.2) .. (0.25,1.5);
    \draw[thick]
          (-0.75,-1.5) .. controls (-0.45,-1.2) and (-0.05,-1.2) .. (0.25,-1.5);

\fill[white]
          (1.65,1.5) .. controls (2.05,1.2) and (2.45,1.4) .. (2.45,1.4) 
         -- (2.45,0) .. controls (2.2,0) and (2.2,0) .. (2.2,0)
          -- cycle;
        \draw[thick]
          (1.65,1.5) .. controls (2.05,1.2).. (2.45,1.4);
        \fill[white]
          (1.65,-1.5) .. controls (2.05,-1.2) and (2.45,-1.4) .. (2.45,1.4)
          -- (2.45,0) .. controls (2.2,0) and (2.2,0) .. (2.2,0)
          -- cycle;
        \draw[thick]
          (1.65,-1.5) .. controls (2.05,-1.2).. (2.45,-1.4);
        \draw[thick] 
          (2.45,1.4) .. controls (1.75, 0.2) and (1.75, -0.2) .. (2.45,-1.4);

  \draw[thick, double distance=5mm] (4.0,0) ++(0:1) arc (0:180:1);
    \draw[thick, double distance=5mm] (4.0,0) ++(180:1) arc (180:360:1);
\fill[white]
        (3.25,1.5) .. controls (2.85,1.2) and (2.45,1.4) .. (2.45,1.4)
        -- (2.45,0) .. controls (2.7,0) and (2.7,0) .. (2.7,0)
        -- cycle;
  \fill[red] (2.45,0) ellipse[x radius=2pt, y radius=5.85pt];
  \draw[thick]
          (3.25,1.5) .. controls (2.85,1.2).. (2.45,1.4);
    \fill[white]
          (3.25,-1.5) .. controls (2.85,-1.2) and (2.45,-1.4) .. (2.45,-1.4)
          -- (2.45,0) .. controls (2.7,0) and (2.7,0) .. (2.7,0)
          -- cycle;
   \draw[thick]
          (3.25,-1.5) .. controls (2.85,-1.2).. (2.45,-1.4);
        \draw[thick] 
          (2.45,1.4) .. controls (3.15, 0.2) and (3.15, -0.2) .. (2.45,-1.4);
          
          \fill[red] (2.45,1.35) ellipse[x radius=2pt, y radius=5.85pt];
          \fill[red] (2.45,-1.35) ellipse[x radius=2pt, y radius=5.85pt];
  
  \draw[->, thick] (0, -2.1) -- (1.2, -3.8);
  \draw[->, thick] (3.7, -2.1) -- (3.2, -3.8);
  
          \draw[thick, double distance=5mm] (1.2,-6.5) ++(0:1) arc (0:180:1);
          \draw[thick, double distance=5mm] (1.2,-6.5) ++(180:1) arc (180:360:1);
          \draw[thick] (3.2, -6.5) ellipse[x radius=21pt, y radius = 58.5pt];
          \draw[thick, dashed] (2.45, -6.5) .. controls (2.93, -6.2) and (3.41, -6.2) .. (3.95, -6.5);
          \draw[thick] (2.45, -6.5) .. controls (2.93, -7.1) and (3.41, -7.1) .. (3.95, -6.5);
          \fill[red] (2.45,-6.5) ellipse[x radius=2pt, y radius=5.85pt];
          
          \fill (2.8, -5.85) ellipse[x radius=2pt, y radius=5.85pt];
          \fill (3, -5.2) ellipse[x radius=2pt, y radius=5.85pt];
          \fill (3.3, -5.6) ellipse[x radius=2pt, y radius=5.85pt];
          \fill (3.3, -7.8) ellipse[x radius=2pt, y radius=5.85pt];
        
          \fill (2.8+0.4, 6.4-5.85) ellipse[x radius=2pt, y radius=5.85pt];
          \fill (3+0.4, 6.4-5.2) ellipse[x radius=2pt, y radius=5.85pt];
          \fill (3.3+0.4, 6.4-5.6) ellipse[x radius=2pt, y radius=5.85pt];
          \fill (3.3+0.4, 6.8-7.8) ellipse[x radius=2pt, y radius=5.85pt];

          \fill (0.5, -5.6) ellipse[x radius=2pt, y radius=5.85pt];
          \fill (1, -5.5) ellipse[x radius=2pt, y radius=5.85pt];
          \fill (0.5-2, 6.7-5.6) ellipse[x radius=2pt, y radius=5.85pt];
          \fill (1, 4.5-5.5) ellipse[x radius=2pt, y radius=5.85pt];
\end{tikzpicture}

   }
   
    \caption{The clutching morphism $\kappa_{g_1, g_2}^\mathrm{u}$}
  \end{subfigure}
  \hspace{1cm}
   \begin{subfigure}[t]{0.24\textwidth} \label{fig: ur 2}
    \centering
    \scalebox{0.7}{
\begin{tikzpicture}[yscale=cos(70)]

  \draw[thick, double distance=5mm] (1.2,-1) ++(0:1) arc (0:180:1);
  \draw[thick, double distance=5mm] (1.2,-1) ++(180:1) arc (180:360:1);

  \draw[thick, double distance=5mm] (1.2,1.5) ++(0:1) arc (0:180:1);
  \draw[thick, double distance=5mm] (1.2,1.5) ++(180:1) arc (180:360:1);

  \draw[thick, double distance=5mm] (3.6,0.2) ++(0:1) arc (0:180:1);
  \draw[thick, double distance=5mm] (3.6,0.2) ++(180:1) arc (180:360:1);
  
  \fill[red] (2.4,0.8) ellipse[x radius=2pt, y radius=5.85pt];
  \fill[red] (2.4,-0.4) ellipse[x radius=2pt, y radius=5.85pt];
  
  \draw[->, thick] (1.2, -3.1) -- (1.2, -4.5);
  \draw[->, thick] (3.7, -2.1) -- (3.2, -3.8);
  
  \draw[thick, double distance=5mm] (1.2,-6.5) ++(0:1) arc (0:180:1);
  \draw[thick, double distance=5mm] (1.2,-6.5) ++(180:1) arc (180:360:1);
  \draw[thick] (3.2, -6.5) ellipse[x radius=21pt, y radius = 58.5pt];
  \fill[red] (2.45,-6.5) ellipse[x radius=2pt, y radius=5.85pt];

  \draw[thick, dashed] (2.45, -6.5) .. controls (2.93, -6.2) and (3.41, -6.2) .. (3.95, -6.5);
  \draw[thick] (2.45, -6.5) .. controls (2.93, -7.1) and (3.41, -7.1) .. (3.95, -6.5);

  \fill (2.8, -5.85) ellipse[x radius=2pt, y radius=5.85pt];
  \fill (3, -5.2) ellipse[x radius=2pt, y radius=5.85pt];
  \fill (3.3, -5.6) ellipse[x radius=2pt, y radius=5.85pt];
    \fill (3.3, -7.8) ellipse[x radius=2pt, y radius=5.85pt];

  \fill (2.8+0.2, 6.6-5.85) ellipse[x radius=2pt, y radius=5.85pt];
  \fill (3+0.2, 6.6-5.2) ellipse[x radius=2pt, y radius=5.85pt];
  \fill (3.3+0.2, 6.6-5.6) ellipse[x radius=2pt, y radius=5.85pt];
  \fill (3.3+0.4, 6.8-7.8) ellipse[x radius=2pt, y radius=5.85pt];
 
\end{tikzpicture}

    }
    \caption{Special case of $\kappa^{\mathrm u}_{1, g_2}$ whose image lies in $\delta_{1, g_2}$}
    \end{subfigure}
    \hfill

    \begin{subfigure}[t]{0.34\textwidth} \label{fig: ramified}
    \centering
    \scalebox{0.7}{
    \begin{tikzpicture}[yscale=cos(70)]

  \draw[thick, double distance=5mm] (-1.2,0) ++(0:1) arc (0:180:1);
  \draw[thick, double distance=5mm] (-1.2,0) ++(180:1) arc (180:360:1);

  \draw[thick, double distance=5mm] (1.2,0) ++(0:1) arc (0:180:1);
  \draw[thick, double distance=5mm] (1.2,0) ++(180:1) arc (180:360:1);

  \fill[white]
    (-0.5,1.5) .. controls (-0.2,1.2) and (0.2,1.2) .. (0.5,1.5)
    -- (0.5,-1.5) .. controls (0.2,-1.2) and (-0.2,-1.2) .. (-0.5,-1.5)
    -- cycle;

  \draw[thick]
    (-0.5,1.5) .. controls (-0.2,1.2) and (0.2,1.2) .. (0.5,1.5);
  \draw[thick]
    (-0.5,-1.5) .. controls (-0.2,-1.2) and (0.2,-1.2) .. (0.5,-1.5);
    
  \draw[thick, double distance=5mm] (3.7,0) ++(0:1) arc (0:180:1);
  \draw[thick, double distance=5mm] (3.7,0) ++(180:1) arc (180:360:1);
  \fill[red] (2.45,0) ellipse[x radius=2pt, y radius=5.85pt];
  
  \draw[->, thick] (0, -2.1) -- (1.2, -3.8);
  \draw[->, thick] (3.7, -2.1) -- (3.2, -3.8);
  
  \draw[thick, double distance=5mm] (1.2,-6.5) ++(0:1) arc (0:180:1);
  \draw[thick, double distance=5mm] (1.2,-6.5) ++(180:1) arc (180:360:1);
  \draw[thick] (3.2, -6.5) ellipse[x radius=21pt, y radius = 58.5pt];
  \fill[red] (2.45,-6.5) ellipse[x radius=2pt, y radius=5.85pt];

  \draw[thick, dashed] (2.45, -6.5) .. controls (2.93, -6.2) and (3.41, -6.2) .. (3.95, -6.5);
  \draw[thick] (2.45, -6.5) .. controls (2.93, -7.1) and (3.41, -7.1) .. (3.95, -6.5);
  
  \fill (2.8, -5.85) ellipse[x radius=2pt, y radius=5.85pt];
  \fill (3, -5.2) ellipse[x radius=2pt, y radius=5.85pt];
  \fill (3.3, -5.6) ellipse[x radius=2pt, y radius=5.85pt];

  \fill (2.8+0.2, 6.4-5.85) ellipse[x radius=2pt, y radius=5.85pt];
  \fill (3+0.2, 6.4-5.2) ellipse[x radius=2pt, y radius=5.85pt];
  \fill (3.3+0.2, 6.4-5.6) ellipse[x radius=2pt, y radius=5.85pt];
  
  \fill (0.5, -5.6) ellipse[x radius=2pt, y radius=5.85pt];
  \fill (0.5-2, 6.7-5.6) ellipse[x radius=2pt, y radius=5.85pt];
\end{tikzpicture}
}
    \caption{The clutching morphism $\kappa_{g_1, g_2}^\mathrm{r}$}
  \end{subfigure}
  
  \caption{Clutching morphisms on the boundary of $\overline{M}_{E, g}$.}
  \label{fig:clutching}
\end{figure}

In Definitions~\ref{def: clutching u} and~\ref{def: clutching r}, we use $Q$ (resp.~$R$) to denote an unbranched point (resp.~branch point) of the target curve, also called the base curve, and $q$ (resp.~$r$) to denote an unramified point (resp.~ramification point) of the source curve, also called the covering curve. In the notation introduced at the beginning of Section~\ref{section:construct_meg_stable}, $Q$ lies on $\Sigma_{\ub}$ and $R$ lies on $\Sigma_{\br}$. Moreover, in the definitions below, we describe the clutching maps on the smooth loci of $\mt_{E,g_1}$ and $\hti_{g_2}$ to simplify the exposition. These clutching maps can be extended to the closures by replacing the source curve~$E$ with a curve $C$ in which one component is $E$ and all other components are copies of $\mathbb{P}^1$, and by replacing the base curve $\mathbb{P}^1$ with a bubble of copies of $\mathbb{P}^1$. Definition~\ref{def: clutching u} deals with the clutching maps depicted in parts (a) and (b) of Figure~\ref{fig:clutching}, where the clutching point is not a ramified point of the double cover.

\begin{definition}\label{def: clutching u}

Given $g_1 \geq 1$, $g_2 \geq 0$, we have the following clutching map:
\begin{align*}
\begin{split}
    \kappa^{{\rm u}}_{g_1,g_2}: & \mt_{E,g_1; 1} \times \hti_{g_2; 1} \to \mt_{E,g_1+g_2+1} \\
     \text{source curve: }&(D,q_1,q_2) \times (H,q_1',q_2') \mapsto (D,q_1,q_2) \cup_{q_1=q_1',q_2=q_2'}(H,q_1',q_2') \\
     \text{target curve: }&(E,Q) \times (\Po, Q' ) \mapsto (E,Q) \cup_{Q=Q'}(\Po,Q'). \\
\end{split}    
\end{align*}
    Here $\pi_1: D \to (E,Q)$ is a double cover parametrized by $\mt_{E,g_1; 1}$ and $\pi_2: H \to (\Po,Q')$ is a double cover parametrized by $\hti_{g_2; 1}$, Moreover, write $\{q_1,q_2\}=\pi_1^{-1}(Q)$ and $\{q_1',q_2'\}=\pi_2^{-1}(Q')$. The output of the clutching map is a double cover of $(E,Q) \cup_{Q=Q'}(\Po,Q')$. Denote the reduced image of $\kappa_{g_1,g_2}^{{\rm u}}$ as $\Du_{g_1,g_2}$. 
    \end{definition}
    
    Notice that $g_2$ can be $0$ in this case, since a $\Z/2\Z$-cover of a genus-$0$ curve has $2$ branching points, and so the target curve has three special points, which is stable. 

    The structure and the generic curve of $\Du_{g_1,g_2}$ depend on $g_1$ in the following ways:
    \begin{enumerate}
    \item Suppose $g_1=1$.
     Let $(H,q_1, q_2, r_1, \cdots, r_{2g_2+2})$ be a generic curve of $\hti_{g_2;1}$, where $q_1,q_2$ are the preimage of the $1$ marked point on $\Po$, and $r_1, \cdots, r_{2g_2+2}$ are the ramification points of~$H$. A generic curve of~$\mt_{E,1;1}$ is either the disconnected $(E,Q) \sqcup (E,Q)$ or an \'etale double cover $(E', q_1,q_2)$ of $(E,q)$.  Therefore, there are two types of configurations for a generic curve of~$\Du_{1,g_2}$.
     We denote the sublocus of~$\Du_{1,g_2}$ consisting of curves of compact type, meaning their dual graph is acyclic, by $\Duct_{1,g_2}$, and we denote the complement of $\Duct_{1,g_2}$ in $\Du_{1,g_2}$ by $\xi_{1,g_2}$. There are two possibilities for the source curve:
    \begin{enumerate}
    \item 
    $(E,Q) \cup_{Q=q_1'}(H,q_1', r_1, \cdots, r_{2g_2+2}, q_2') \cup_{q_2'=Q}(E,Q)$. In this case it is of compact type and is in $\Duct_{1,g_2}$.  This corresponds to Figure~\ref{fig:clutching}.(b).
    \item $(E',q_1,q_2) \cup_{q_1=q_1',q_2=q_2'}(H,q_1', q_2', r_1, \cdots, r_{2g_2+2})$. In this case the curve has toric rank $1$ and is in $\xi_{1,g_2}$. This corresponds to Figure~\ref{fig:clutching}.(a).
    \end{enumerate}
\item Suppose $g_1 \geq 2$. Then $\mt_{E,g_1;1}$ is connected  and only the clutching map depicted in Figure~\ref{fig:clutching}.(a) can occur. Let $(D, q_1, q_2)$ be a generic curve of $\mt_{E,g_1;1}$ and $(H, q_1', q_2')$ be a generic curve of $\hti_{g_2,1}$. Then a generic curve of $\Du_{g_1,g_2}$ is of the form $(D, q_1, q_2) \cup_{q_1=q_1',q_2=q_2'}(H, q_1, q_2)$. Therefore, a curve in $\Du_{g_1,g_2}$ has toric rank at least $1$.
\end{enumerate}

Next, we define the clutching map in the case when the clutching point is also a ramified point of the double cover, as is depicted in Figure~\ref{fig:clutching}.(c).

    \begin{definition} \label{def: clutching r}
    Given $g_1 \geq 2$, $g_2\geq 1$, we have the following clutching map:
    \begin{align*}
    \begin{split}
    \kappa^{{\rm r}}_{g_1,g_2}: & \mt_{E,g_1} \times \hti_{g_2} \to \mt_{E,g_1+g_2} \\
    \text{source curve: }& (D, r_i) \times (H, r'_j) \mapsto (D, r_i)\cup_{r_{2g_1-2}=r'_1}(H, r'_j)\\
    \text{target curve: }& (E, R_i) \times (\Po, R'_j) \mapsto (E, R_i)\cup_{R_{2g_1-2}=R'_1}(\Po, R'_j).
    \end{split}
    \end{align*}
    On the source curves the $(2g_1-2)^{th}$ smooth ramification point of $D$ is identified with the first smooth ramification point of $H$ as a node, which is still a ramification point, and similarly for the target curves. Denote the reduced image of $\kappa_{g_1,g_2}^{{\rm r}}$ as $\Dr_{g_1,g_2}$.
\end{definition}
The generic point of $\Dr_{g_1,g_2}$ is represented by a curve of compact type. 

\begin{lemma}\label{lem: boundary is clutching}
For $g\geq 2$, the boundary of $\mt_{E,g}$ is the union of the image of the above clutching maps. In particular,
\[\mt_{E,g}-\mt_{E,g}^\circ=\bigcup_{g_1=1}^{g-1}\Du_{g_1, g-1-g_1} \cup \bigcup_{g_1=2}^{g-1}\Dr_{g_1,g-g_1}.\]
\end{lemma}
\begin{proof}
Suppose $\pi: D \to C$ is in $\mb_{E,g}-M_{E,g}$. Then it must be in the image of $\mt_{E,g}-\mtc_{E,g}$.
Consider the forgetful morphism $\mt_{E,g}\to \Mb_{1, 2g-2}$ that records the target curve with the $(2g-2)$ smooth branch points.
Since $(C,R_1, \ldots, R_{2g - 2})$ is not smooth, it is represented by a point in $\Mb_{1, 2g-2}-\M_{1, 2g-2}$, with one of its irreducible components being $E$. We know that the boundary of $\Mb_{1, 2g-2}$ lies in the union of the image of the following clutching morphisms:

\begin{align*}
 \kappa_{1,i;0,j}: \Mb_{1, i+1} \times \Mb_{0, j+1} &\to \Mb_{1, i+j} = \Mb_{1,2g-2}, \\
    \kappa_{0}: \Mb_{0, 2g } &\to \Mb_{1, 2g-2}.
\end{align*}
Here $(i,j)$ runs through the pairs $i \geq 0$, $j \geq 2$, $i+j=2g-2$. The map $\kappa_{1,i;0,j}$ clutches two curves together at an ordinary double point, while $\kappa_{0}$ creates a node on one curve by identifying two marked points on it. 

Notice that $(C,R_1, \ldots, R_{2g - 2})$ can not be in the image of $\kappa_{0}$ since $C$ should have $E$ as one of its irreducible components. Hence, there must exist $i \geq 0, j \geq 2$ such that $i+j=2g-2$ and $C= (C_1, R_1', \ldots, R_i', R_{i + 1}')\cup_{R_{i + 1}' = R_{j + 1}''}(C_2, R_1'', \ldots, R_j'', R_{j + 1}'')$ for some $(i + 1)$- and $(j + 1)$-pointed stable curves $(C_1, R_1', \ldots, R_i', R_{i + 1}')$ and $(C_2, R_1'', \ldots, R_j'', R_{j + 1}'')$ of genera $1$ and $0$, respectively. We can further choose $C_1$ and $C_2$ such that $C_2$ is a smooth genus-$0$ curve, since $C_1$ is allowed to be degenerate. Let $D_1=\pi^{-1}(C_1)$ and $D_2=\pi^{-1}(C_2)$ be the preimages of $C_1$ and $C_2$ under $\pi: D \to C$, and let $Q$ be the node on $C$ resulting from clutching. 
Note that $C_2$ has $j+1$ marked points, one of them is $Q$. We have the following possibilities for the configuration of $\pi: D \to C$.
\begin{enumerate}
    \item If $j$ is even: Write $j=2m$. Since $C_2$ is smooth of genus $0$, any $\Z/2\Z$-cover of $C_2$ must have an even number of the branching points. Therefore $Q$ is not a branching point.
    In this case, $\pi: D_2 \to C_2$ is a stable $\Z/2\Z$-cover branched at $j$ points, so this gives an element of $\hti_{m-1;1}$.
    Moreover, $E$ is a component of $C_1$ and $D_1 \to C_1$ is a stable $\Z/2\Z$-cover, branched at $i=2g-2-j=2(g-m-1)$ points, hence this gives an element of $\mt_{E,g-m;1}$. 
    Therefore, $D \to C$ is in $\Du_{g-m,m-1}$. The stability of $D_2$ implies that $j\geq 2$, so $m-1\geq 0$.
    \item If $j$ is odd: Write $j=2m+1$. Then $Q$ must be a branching point. In this case, $\pi: D_2 \to C_2$ is a stable~$\Z/2\Z$-cover branched at $j+1=2(m+1)$ points, so it is in $\hti_{m}$.
    Moreover, $E$ is a component of $C_1$ and $D_1 \to C_1$ is a stable $\Z/2\Z$-cover branched at $i+1=2(g-m-1)$ points, hence it is in $\mt_{E,g-m}$.
    Therefore, $D \to C$ is in $\Dr_{g-m,m}$.
    Since $i$ is odd, we have $i \geq 1$, implying $g-m \geq 2$. The stability of $C_2$ implies that $j\geq 2$, and it is odd, so $j\geq 3$. Thus we conclude~$m \geq 1$.
    \end{enumerate}
    Therefore $\pi: D \to C$ lies in the union of image of clutching maps.
\end{proof} 

\begin{example}[Description of the boundary of $\overline{M}_{E, 2}$]
Let $E $ be an arbitrary elliptic curve as above. 
From our description in Lemma \ref{lem: boundary is clutching}, the boundary $\overline{M}_{E, 2} - {M}_{E, 2}$ of $M_{E, 2}$ consists of curves that are double covers of $$(E, Q)\cup_{Q = Q'} (\Po, Q', R_1, R_2).$$ After applying an automorphism of $\Po$, we can assume $(Q', R_1, R_2)=(\infty, 0, 1)$. Then, the target curve is the clutch of $(E, Q)$ and $(\Po, \infty, 0, 1)$ with $Q \in E$ and $\infty\in \Po$ identified. This double cover is branched at $0$ and $1$. There are two kinds of curves in  $\overline{M}_{E, 2} - {M}_{E, 2}$:
\begin{enumerate}
    \item $D = (E', q_1, q_2)\cup_{q_1 = q_1', q_2 = q_2'} (\Po, q_1', q_2', r_1, r_2)$, where $E'$ is a curve of genus~$1$, such that $\pi: E' \to E$ is an \'etale double cover, $\pi^{-1}(Q) = \{q_1, q_2\}$, and $\pi':\Po \to \Po$ is the double cover branched at $0$ and $1$, with $\pi'^{-1}(\infty) = \{q_1', q_2'\}$, $\pi'^{-1}(0) = r_1$, and $\pi'^{-1}(1) = r_2$. Here $D$ has toric rank $1$.

    \item $D' = (E, Q)\cup_{Q = q_1'} (\Po, q_1', q_2', r_1, r_2)\cup_{q_2' = Q}(E, Q)$, with $ (\Po, q_1', q_2', r_1, r_2)$ as in the first case. Here $D'$ is of compact type. 
\end{enumerate}
We will use this description in Section \ref{Sec: p rank strat MEg}.
\label{example:bdry_of_ME2}
\end{example}

In the preceding example, the explicit descriptions of the boundary components show that they have codimension~$1$ in $\mb_{E,g}$ when $g = 2$. The following corollary demonstrates that the same holds for general~$g$.

\begin{corollary} \label{cor: smooth locus is dense}
For $g \geq 2$, the boundary of $\mb_{E,g}$ has codimension $1$ in $\mb_{E,g}$. Consequently, $M_{E,g}$ is open and dense in $\mb_{E,g}$.
\end{corollary}
\begin{proof}
Let $g_1, g_2 \geq 1$ be integers such that $g_1 + g_2 = g$. We compute 
\begin{align*}
\dim(\Du_{g_1, g_2 - 1}) = \dim(\mt_{E,g_1;1})+\dim(\hti_{g_2 - 1;1}) = (2g_1-2)+1+(2g_2 - 3)+1=\dim(\mb_{E,g})-1, 
\end{align*}
while for $g_1 \geq 2$, we find that
\begin{align*}
\dim(\Dr_{g_1, g_2}) = \dim(\mt_{E,g_1})+\dim(\hti_{g_2}) =(2g_1-2)+(2g_2-1) =\dim(\mb_{E,g})-1. 
\end{align*}
Therefore, all loci $\Du_{g_1, g_2 - 1}$ and $\Dr_{g_1, g_2}$ have codimension $1$ in $\mb_{E, g}$, and Lemma \ref{lem: boundary is clutching} implies that the boundary of $\mb_{E, g}$ has codimension $1$ in $\mb_{E, g}$.  

Suppose that $M_{E,g}$ is not dense in $\mb_{E,g}$. Then there is a non-empty, open set $U \subseteq \mb_{E,g}$ that does not intersect $M_{E,g}$ so that $U \subseteq \mb_{E,g}-M_{E,g}$, and hence $\dim(U)\leq 2g-3$. On the other hand, Theorem \ref{theorem:meg_smooth} says that $\mb_{E,g}$ is irreducible of dimension $2g-2$. This implies $\dim(U)=2g-2$, which is a contradiction. To conclude, $M_{E,g}$ is dense in $\mb_{E,g}$.
\end{proof}

\subsection{The reduced moduli spaces $M_{E,g;n}^{\red} \subset \mb_{E,g;n}^{\red}$}\label{Sec: reduced moduli space}

In this section, we introduce a ``reduced'' moduli space $\mb_{E,g;n}^{\red}$ as a closed substack of $\mb_{E,g;n}$. To construct this space,
we consider the map $\pi: \mb_{E,g;n} \to E$ that sends a cover $D \to C$ to the sum of the branch points of the cover. Since one component of $C$ is identified with $E$ and all other components are copies of $\mathbb{P}^1$, this sum of points indeed lives in $\Jac(C) \cong \Jac(E) \cong E$.

\begin{definition} 
We define $\mb_{E,g;n}^{\red}$ as the fiber of the map $\pi$ over the identity $O_E \in E$.
We also denote the open substack of smooth covers $M_{E,g;n}^{\red} \coloneqq M_{E,g;n} \cap \mb_{E,g;n}^{\red}$.
\end{definition}

\begin{proposition}
$\mb_{E,g;n}^{\red}$ is a proper Deligne-Mumford stack of dimension $2g - 3 + n$.
\label{prop:moduli_red_is_proper}
\end{proposition}

\begin{proof}
The claim that $\mb_{E,g;n}^{\red}$ is proper and Deligne-Mumford follows from the fact that $\mb_{E,g;n}$ is proper and Deligne-Mumford. To compute its dimension, we 
 observe that all the fibers of $\pi: \mb_{E,g;n} \to E$ are isomorphic. To get an isomorphism between $\pi^{-1}(O_E) = \mb_{E,g;n}^{\red}$ and $\pi^{-1}(x)$, take a point $y \in E$ such that $(2g - 2) \cdot y = x$ and use translation by $y$ to get an isomorphism between $\pi^{-1}(O_E)$ and~$\pi^{-1}(x)$. Because $\mb_{E,g;n}$ has dimension $2g - 2 + n$ (Theorem \ref{theorem:meg_smooth}) and all the fibers of $\pi$ are isomorphic, it follows that $\pi^{-1}(O_E) = \mb_{E,g;n}^{\red}$ has dimension $2g - 3 + n$, as desired.
\end{proof}

One intuitive reason to consider $\mb_{E,g;n}^{\red}$ is that we ideally would like to identify double covers of $E$ if they are the same after composing with an automorphism of $E$. Because $\Aut(E)$ is an extension of a finite constant group scheme by $E$ (where $E$ acts by translations), taking the quotient by the automorphism group of $E$ will reduce the dimension of our moduli space by 1. While $\mb_{E,g;n}^{\red}$ is not quite a quotient of $\mb_{E,g;n}$, it is a slice for the action of $\Aut(E)$. 
Therefore, $\mb_{E,g;n}^{\red}$ maps finitely and surjectively to $[\mb_{E,g;n}/\Aut(E)]$. The reason we use $\mb_{E,g;n}^{\red}$ rather than the quotient $[\mb_{E,g;n} / \Aut(E)]$ is that $\mb_{E,g;n}^{\red}$ is a substack of $\mb_{E,g;n}$, and hence it inherits the structure of a proper Deligne-Mumford stack. Furthermore, it is more straightforward to define the morphism in Proposition~\ref{forgetful map is quasi-finite} for $\mb_{E,g;n}^{\red}$ than for $[\mb_{E,g;n}/\Aut(E)]$.

\begin{proposition}\label{forgetful map is quasi-finite}
    We have the following morphism: 
    \[\mb_{E,g}^{\red} \xrightarrow{\varphi} \Mb_{g,2g-2} \xrightarrow{\pi} \Mb_{g}\]
    The first map $\varphi $ sends a cover $D \to C$ to its source curve $D$, with the $2g - 2$ ramification points as marked points.  The second map $\pi$ forgets the marked points and stabilizes the curve. Then $\pi \circ \varphi$ is quasi-finite on the smooth locus of $\mbr_{E,g}$ and generically finite on each boundary component of~$\mbr_{E,g}$.

\end{proposition}
\begin{proof}
This can be proved using a careful case-by-case analysis, as is done in the appendix, Proposition~\ref{forgetful map is quasi-finite - appendix}. 
\end{proof}

\begin{remark}
The preceding proposition is consistent with \cite[Theorem 3.7]{schmittvanzelm} and 
\cite[Proposition 3.13]{schmittvanzelm}, which describe the dimension of boundary components of the whole locus of Galois $G$-covers in $\overline{\mathcal{M}}_g$. 
\end{remark}

\section{The $p$-rank stratification of $\B_{E, g}$}\label{Sec: p rank strat MEg}

Let $E = (E, O_E)$ be a fixed (smooth) elliptic curve over an algebraically closed field $k$ of characteristic~$p>2$ with $p$-rank $f_E$. Recall that there are two possibilities: $f_E = 1$ if $E$ is ordinary, and $f_E = 0$ if $E$ is supersingular.  
Consider the locus $\B_{E, g}$ of smooth curves of genus $g\geq 2$ that are double covers of $E$. Our goal is to describe the $p$-rank stratification of $\B_{E, g}$. We achieve this by analyzing its closure in~$\overline{\M}_g$ using techniques similar to those in \cite{faber2004complete}, \cite{glass2005hyperelliptic}, and \cite{Achter_2011}. As demonstrated in the proof of Theorem \ref{theorem:p_rank_stratification_MEg}, the cases when $E$ is ordinary and when $E$ is supersingular differ significantly, with the latter case requiring particular attention. 

A smooth curve $D$ of genus $g \geq 2$ is called bielliptic if there exists an elliptic curve $E'$ and a double cover $D \to E'$. We denote by $\mathcal{B}_g$ the locus of bielliptic curves in $\mathcal{M}_g$. Given an elliptic curve~$E'$, $\mathcal{B}_{E', g}$ can be viewed as a subset of $\mathcal{B}_g$.  
Consequently, we can describe the $p$-rank stratification of the entire locus $\mathcal{B}_g$ of bielliptic curves of genus $g$.

As we have seen in Proposition \ref{forgetful map is quasi-finite}, there is a morphism $\mb_{E, g}^{\red} \to \overline{\mathcal{M}}_g$ that is quasi-finite on the smooth locus and generically quasi-finite on every boundary component.
Denote by $\Bb_{E, g}$ the image of this morphism and note that $\B_{E,g}$ coincides with the image of $M_{E,g}^{\red}$. Since $\mb_{E, g}^{\red}$ is a proper Deligne-Mumford stack by Proposition \ref{prop:moduli_red_is_proper} and $M_{E,g}^{\red} \subseteq \mb_{E, g}^{\red}$ is open and dense, it follows that $\Bb_{E, g} \subseteq \overline{\mathcal{M}}_g$ is the closure of $\B_{E, g} \subset \mathcal{M}_g$. We summarize some properties of $\B_{E,g}$ and  $\overline{\B}_{E, g}$ below.

\begin{lemma}
\label{lem:Bbeq_is_proper_and_irreducible}
For any $g \geq 2$, the moduli space $\overline{\mathcal{B}}_{E,g}$ is irreducible, proper, and has dimension~$2g - 3$.
\end{lemma}
\begin{proof}
Note that $\Bb_{E,g}$ is irreducible for any $g \geq 2$,  because the map $\mbr_{E,g}\to \Bb_{E,g}$ factors through $[\mb_{E,g}/\Aut(E)]$ and it was shown in Lemma \ref{Lem: Meg is irr} that $\mb_{E,g}$ is irreducible. The remaining properties follow from Propositions \ref{prop:moduli_red_is_proper} and \ref{forgetful map is quasi-finite}.
\end{proof}

\begin{lemma}
\label{lem:Beq_is_quasi_affine_and_irreducible}
For any $g \geq 2$, the moduli space $\mathcal{B}_{E,g}$ is irreducible, has dimension $2g - 3$, and does not contain any positive-dimensional complete subvarieties.
\end{lemma}
\begin{proof}
Lemma \ref{lem:Bbeq_is_proper_and_irreducible} and Proposition \ref{forgetful map is quasi-finite} imply that $\B_{E,g}$ is irreducible. 
Moreover, note that the coarse moduli space of $\B_{E,g}$ is quasi-affine, and consequently, $\mathcal{B}_{E,g}$ cannot contain positive-dimensional complete subvarieties. Indeed, let $\mt_{E,g}^{\circ}$ denote the open locus of $\mt_{E,g}$ where the source curve is smooth. Then $\B_{E,g}$ is the image of $[\mt_{E,g}^{\circ}/E]$ under a quasi-finite and proper morphism (hence it induces a finite morphism on the coarse moduli space). The morphism is quasi-finite according to Proposition \ref{forgetful map is quasi-finite}, and it is proper because it is the restriction of $[\mt_{E,g}/E] \to \Bb_{E,g}$, which are both proper. Hence, it suffices to show that the coarse moduli space of $[\mt_{E,g}^{\circ}/E]$ is quasi-affine. If we let $\mathrm{PConf}^{2g - 2}(E)$ be the ordered configuration space of $2g - 2$ points on $E$, then there is a proper \'{e}tale map $\mt^{\circ}_{E,g} \to \mathrm{PConf}^{2g - 2}(E)$ that records the branch points of a cover, inducing a finite map on coarse moduli spaces. Now $$\mathrm{PConf}^{2g - 2}(E) \cong E \times \mathrm{PConf}^{2g - 3}(E - \{O_E\}) \text{ via } (x_1,\ldots,x_{2g - 2}) \leftrightarrow (x_1,(x_2 - x_1,\ldots,x_{2g - 2} - x_1)),$$ which implies  $\mathrm{PConf}^{2g - 2}(E)/E \cong \mathrm{PConf}^{2g - 3}(E~-~\{O_E\})$. Finally, $$[\mathrm{Conf}^{2g - 2}(E)/E] \cong [\mathrm{PConf}^{2g - 2}(E)/(E \times S_{2g - 3})] \cong \mathrm{PConf}^{2g - 3}(E - \{O_E\})/S_{2g - 3}.$$
     The coarse space of $\mathrm{PConf}^{2g - 3}(E - \{O_E\})/S_{2g - 3}$ is affine because $E - \{O_E\}$ is affine and thus the coarse spaces of $[\mathrm{Conf}^{2g - 2}(E)/E]$ and $[\mt^{\circ}_{E,g}/E]$ are affine.    
\end{proof}

\subsection{The boundary of $\B_{E, g}$}

Let $\Delta = \Mb_g - \M_g = \Delta_0 \cup \Delta_1 \cup \ldots \cup \Delta_{\lfloor g/2 \rfloor}$,
denote the boundary of $\Mb_g$, with $\Delta_i$ its irreducible components. The generic point of $\Delta_0$ corresponds to an irreducible curve $C_0$ with an ordinary double point, whose normalization $\Tilde{C}_0$ is a genus-$(g - 1)$ curve, while the generic point of $\Delta_i$ corresponds to a reducible curve with components of genus $i$ and $g - i$ for any $1\leq i \leq \left \lfloor \frac{g }{2} \right \rfloor$.
By abuse of notation, we use $\Du_{g_1, g_2 - 1}$, $\Dr_{g_1, g_2}$ for $g_1, g_2 \geq 0$ with $g_1 + g_2 = g$ to denote the image of $\Du_{g_1, g_2 - 1}$, $\Dr_{g_1, g_2}$ under the forgetful morphism $\mb_{E, g}^{\red} \to \overline{\mathcal{M}}_g$.

It follows from the description in Section \ref{section:boundary_description} that each boundary component consists of isomorphism classes of curves that lie in the closure of a locus of double covers of the configuration 
\begin{equation}
    (E, O_E)\cup_{O_E = \infty} (\Po, \infty)
    \label{eqn:double_covers_bdry}
\end{equation}
with exactly $2g - 2$ distinct smooth branch points and that curves corresponding to elements in the loci $\Du_{g_1, g_2 - 1}$ (resp. $\Dr_{g_1, g_2}$) are unramified (resp. ramified) above the node.
We see that $\Bb_{E, g} - \B_{E, g}$ consists of the union of the loci $\Du_{g_1, g_2 - 1}$ and $\Dr_{g_1, g_2}$. The dimensions of these loci equal $2g - 4$ by Proposition \ref{forgetful map is quasi-finite}.
Let 
$\Duct_{1, g-2} \subseteq \Du_{1, g-2}$ 
denote the union of the components of $\Du_{1, g-2}$ whose generic points are each represented by a curve of compact type,
and denote 
$\xi_{1, g-2} = \Du_{1, g-2} - \Duct_{1, g-2}$.
We have the following relations: 
$$\xi_{1, g - 2} \subseteq \Delta_0 \text{ and } \Du_{g_1, g_2 - 1} \subseteq \Delta_0, \text{ for all }2 \leq g_1 \leq g - 1 \text{ satisfying } g_1 + g_2 = g;$$
$$ \Duct_{1, g-2}, \Dr_{g - 1, 1} \subseteq \Delta_1 \quad \text{and} \quad \Dr_{i, g - i} \subseteq \Delta_{\min\{i,g-i\}}, \text{ for all } 2\leq i \leq g - 2.$$

\begin{example}[Description of the boundary $\Bb_{E, 2} - \B_{E, 2}$]
Let $E = (E, O_E)$ be an arbitrary elliptic curve as above, and let $f_E$ be its $p$-rank. In Example \ref{example:bdry_of_ME2}, we describe the boundary of $\overline{M}_{E, 2} - {M}_{E, 2}$ of $M_{E, 2}$, which we use to describe its image 
in $\Mb_2$ here. The locus $\Bb_{E, 2} - \B_{E, 2}$ is $0$-dimensional and it consists of isomorphism classes of curves that are double covers of 
$$(E, O_E)\cup_{O_E = \infty} (\Po, \infty),$$ 
ramified above $0, 1 \in \Po$ and unramified otherwise. We find two kinds of these curves (which represent elements of $\xi_{1, 0} \subseteq \Delta_{0}$ and $\Duct_{1, 0} \subseteq \Delta_1$ respectively):
\begin{enumerate}
    \item $D_1 = (E', q_1, q_2)/\{q_1 = q_2\},$ where $\pi: E' \to E = (E, O_E)$ is an \'etale double cover, and $q_1, q_2 \in E'$ two distinct points lying above $O_E \in E$, i.e., $\pi^{-1}(O_E) = \{q_1, q_2\}$; or   
    \item $D_2 = (E, O_E)\cup_{O_E = O_E} (E, O_E)$, where two copies of $(E, O_E)$ are clutched together with their marked points $O_E$ being identified. 
\end{enumerate}
We note that there are exactly $3$ curves of the first kind and a unique curve of the second kind. Namely, the \'etale double covers of $E$ correspond to the choices of $2$-torsion points on $E$. Since $\#E[2](k) = 4$, there are $3$ non-trivial $2$-torsion points on $E$ defining different \'etale double covers $E' \to E$, while $O_E$ is the fourth $2$-torsion point, and it defines a trivial double cover $(E, O_E)\sqcup (E, O_E) \to (E, O_E)$.  Since $f_{E'} = f_{E}$, note that $f_{D_1} = f_{E} + 1$ (see \cite[Section 9.2, Example 8]{blr}), while $f_{D_2} = 2\cdot f_E$.
\label{example:bdry_of_ME2_in_M2}
\end{example}

\subsection{Computing the dimensions of the $p$-rank strata of $\Bb_{E, g}$}

Let $V_f(\Bb_{E, g}) \subseteq \Bb_{E, g}$ denote the $p$-rank $\leq f$ locus of $\Bb_{E, g}$.  
To obtain a lower bound on $\dim V_f(\Bb_{E, g})$ for $g \geq 2$, we use the purity of the $p$-rank stratification, a standard tool in this context. We rely on a version by Faber and van der Geer, based on a result of Oort, later generalized to the Newton polygon stratification as de Jong-Oort's purity.

\begin{lemma}[{\cite[Lemma 2.1]{faber2004complete}, \cite[1.6]{Oort1974}, \cite[Theorem 4.1]{de2000purity}}]\label{Lem: purity we are using}
Let $B$ be an irreducible scheme over $k$ and let $\mathcal{D} \to B$ be a family of stable curves over $B$. Let $f$ be the $p$-rank of the generic fiber $D_\eta$. Then the closed subset of $B$ over which the $p$-rank of the fiber is at most $f - 1$ is either empty or pure of codimension $1$ in $B$.
\end{lemma}

Let $\overline{\mathcal{H}}_g$ denote the locus of stable hyperelliptic curves of genus $g$, and let $V_f(\overline{\mathcal{H}}_g) \subseteq \overline{\mathcal{H}}_g$ be its 
$p$-rank $\leq f$ locus. 
Throughout the section, we will need the following result by Glass and Pries.

\begin{lemma}[{\cite[Theorem 1, Proposition 2]{glass2005hyperelliptic}}]
\label{lemma:glass_pries}
Let $g \geq 2$. For any $0\leq f \leq g$, the locus $V_f(\overline{\mathcal{H}}_g)$ is pure of dimension $g - 1 + f$ in $\overline{\mathcal{H}}_g$. 
\end{lemma}

We start by computing the dimensions of the loci $V_f(\Bb_{E, g})$ for $g = 2$.

\begin{lemma}
\label{lemma:ME2}
The locus $V_f({\Bb}_{E, 2})$ is pure of dimension $f - f_E$ for any $f_E \leq f \leq 1 + f_E$.
\end{lemma}

\begin{proof}
Let $E$ be an elliptic curve in characteristic $p>2$ with $p$-rank $f_E$. Recall that ${\Bb}_{E, 2}$ is irreducible and proper by Lemma \ref{lem:Bbeq_is_proper_and_irreducible}, while $\B_{E, 2}$ cannot contain complete subvarieties by Lemma~\ref{lem:Beq_is_quasi_affine_and_irreducible}. Therefore, $\Bb_{E, 2}-{\B}_{E, 2} \neq \varnothing$. One boundary point of $\Bb_{E, 2}-{\B}_{E, 2}$ corresponds to a curve~${D' = (E', q_1, q_2)/\{q_1 = q_2\}}$ as in Example \ref{example:bdry_of_ME2_in_M2} with $f_{D'} = f_E + 1$ (computed using \cite[Section 9.2, Example 8]{blr}). 
Therefore, by the purity result (Lemma~\ref{Lem: purity we are using}), the generic $p$-rank of ${\Bb}_{E, 2}$ equals $f_E + 1$. 

It remains to show that $V_{f_E}(\Bb_{E, 2})$ is non-empty. By Example \ref{example:bdry_of_ME2_in_M2}, $D = (E, O_E)\cup_{O_E = O_E}(E, O_E)$, with $f_D = 2f_E$, represents a curve lying in the boundary of ${\Bb}_{E, 2}$. If $f_E = 0$, then $f_D = 0$. Otherwise, assume $f_E = 1$. The locus of ordinary curves in ${\Bb}_{E, 2}$ is quasi-affine; this follows from \cite[6.1]{Oort1999}. Therefore, ${\Bb}_{E, 2}$ contains an element that corresponds to a curve $D$ with lower $p$-rank, $f_D = f_E = 1$. Moreover, $D$ is a smooth curve. This follows from Example \ref{example:bdry_of_ME2_in_M2}, using that both $D_1$ and $D_2$ have $f_{D_1} = f_{D_2} = 2$ when $f_E = 1$.
\end{proof}

\begin{remark}
Let $f_E = 0$.
By \cite[Corollary 1]{mumford}, given a double cover $D \to E$ and its Prym variety $E'$, there is a double cover $D \to E'$, and $\mathrm{Jac}(D) \sim E \times E'$, an isogeny of degree $2$. Since $p \neq 2$, this shows that every smooth curve representing an element of $V_0(\Bb_{E, 2})$ must be a superspecial curve. Because there are only finitely many smooth superspecial curves of genus $g$ in characteristic~$p$, this provides an alternative proof that $\dim V_0(\Bb_{E, 2}) = 0$.
\label{remark:superspecial_genus2}
\end{remark}

\begin{corollary}
Let $E$ be an elliptic curve in characteristic $p>3$ with $p$-rank $f_E$. For any $f_E \leq f \leq f_E + 1$, there exists a smooth curve $D$ of genus $2$ with $p$-rank $f_D = f$ which is a double cover of $E$.
\label{cor:smooth_genus2}
\end{corollary}

\begin{proof}
It is enough to consider the case $f_E = 0, f_D = 0$ since the proof of the preceding lemma implies the result in other situations. Write $E: y^2 = x(x - 1)(x - \lambda)$. For $p> 3$, there is $\lambda' \neq \lambda$ such that $E': y^2 = x(x - 1)(x - \lambda')$ is supersingular; see \cite[V.4]{silverman}. Consider the double covers $E \to \mathbb{P}^1$ and $E' \to \mathbb{P}^1$ given by the inclusion $k(x) \subset k(x, y)$ of function fields, and the fiber product $E \times_{\mathbb{P}^1} E'$ taken with respect to these maps.
The normalization of $E\times_{\mathbb{P}^1}E'$ is a supersingular genus~$2$ curve with a map of degree~$2$ to $E$ and the result follows.
\end{proof}

\begin{remark}
\label{remark:ekedahl_p=3}
The first part of the proof of Lemma \ref{lemma:ME2} establishes that there is a smooth curve $D$ of genus $2$ which is a double cover of an ordinary elliptic curve $E$ in characteristic $3$. However, if $E$ is a supersingular elliptic curve in characteristic $3$, the locus $V_0 (\B_{E, 2})$ of smooth curves $D$ in $\B_{E, 2}$ with $f_D = 0$ is empty. This follows from \cite[Theorem~1.1]{ekedahl}, which states that there are no smooth superspecial curves of genus~$2$ in characteristic $3$ and Remark \ref{remark:superspecial_genus2}.
\end{remark}

The following lemmas are needed to determine the dimension of $V_f(\overline{\B}_{E,g})$ for general $g$. 

\begin{lemma}
\label{lemma:prankbdry_res1}
For any $g \geq 2$ and $2f_E \leq f \leq g - 2 + 2f_E$ (resp. $f_E + 1 \leq f \leq g - 1 + f_E$), we have $$\dim V_{f}(\Duct_{1, g-2}) =  g - 2 + f - 2f_E \text{ (resp. }\dim V_{f}(\xi_{1, g - 2}) =  g - 3 + f - f_E \text{)}.$$  
\end{lemma}
\begin{proof}
 For $g\geq 4$, elements of $V_f(\Du_{1, g - 2})$ correspond to double covers of $(E, O_E)\cup_{O_E, \infty} (\Po, \infty)$, ramified above distinct points $R_i\in \Po - \{\infty\}, 1\leq i \leq 2g - 2$ and unramified otherwise, with $p$-rank at most $f$. Each such element corresponds to one of the following two types of clutches:
 \begin{enumerate}
     \item \label{itm:prob}  $(E, O_E)\cup_{O_E = q_1} (H, q_1, q_2)\cup_{q_2 = O_E} (E, O_E)$ representing an element of $V_{f}(\Duct_{1, g-2})$, 
where $H$ represents an element of $V_{f - 2f_E}(\overline{\HH}_{g - 2})$, and 
$q_1, q_2 \in H$ are two distinct points that are in the same orbit under the involution of $H$.
     \item  $(E', q_1', q_2')\cup_{q_1' = q_1, q_2' = q_2} (H, q_1, q_2)$ representing an element of $V_{f}(\xi_{1, g - 2})$,
where $H$ is as in the previous case, and $E' \to E$ an \'etale double cover with $q_1', q_2' \in E'$ two distinct points in the same orbit under the induced bielliptic involution. 
 \end{enumerate}
The dimension of the locus of curves of the first form is 
$\dim V_{f - 2f_E}(\overline{\HH}_{g - 2, 1}) = g - 2 + f - 2f_E$, while for the second form, it is $\dim V_{f - f_E - 1}(\overline{\HH}_{g - 2, 1}) = g - 3 + f - f_E$. 
For $g = 2$, the result follows from the proof of Lemma \ref{lemma:ME2}, while for $g = 3$, it follows similarly as for $g \geq 4$. The only difference is that $H = (H, O_H)$ is an elliptic curve, and $q_1, q_2 \in H - \{O_H\}$ are two distinct points mapping to the same point under the double cover $H \to \Po$ induced by the divisor $2 \cdot O_H$. Thus, instead of $\dim V_{f - 2f_E}(\overline{\HH}_{g - 2, 1})$ and $\dim V_{f - f_E - 1}(\overline{\HH}_{g - 2, 1})$, we consider $\dim V_{f - 2f_E}(\overline{\mathcal{M}}_{1, 2})$ and $\dim V_{f - f_E - 1}(\overline{\mathcal{M}}_{1, 2})$ in the dimension computations.
\end{proof}

\begin{lemma}
\label{lemma:prankbdry_res2}
Let $g \geq 3$ and let $f_E \leq f \leq g - 1 + f_E$.
Assume that for all $2\leq g' < g$ and $1\leq f' \leq f$, $V_{f'}(\overline{\B}_{E, g'})$ is pure of dimension $g' - 2 + f' - f_E$. Then, for any $g_1 \geq 2$ and $g_2 \geq 1$ such that $g_1 + g_2 = g$ the loci $V_{f}(\Dr_{g_1, g_2})$ and $V_{f}(\Du_{g_1, g_2 - 1})$ are all pure of dimension $g - 3 + f - f_E$.    
\end{lemma} 
\begin{proof}
 Let us fix $f_E \leq f \leq g - 1 + f_E$, $g_1 \geq 2$ and $g_2 \geq 1$ such that $g_1 + g_2 = g$. (Note that the outcomes of the computations below are also valid in the cases $g_1 = g - 1$ and $g_1 = g - 2$ since $\dim V_0(\overline{\mathcal{M}}_{0, 3}) = 0$ and $\dim V_{f_2}(\overline{\mathcal{M}}_{1, 1}) = f_2$, and using the identifications $\overline{\HH}_{0, 1} \cong \overline{\mathcal{M}}_{0, 3}$ and $\overline{\HH}_{1} \cong \overline{\mathcal{M}}_{1, 1}$ similarly as in Section \ref{section:boundary_description} and the proof of Lemma~\ref{lemma:prankbdry_res1}.)

 We first compute the dimension of  $V_{f}(\Dr_{g_1, g_2})$. Let $f_1$ and $f_2$ be numbers that satisfy $f_E\leq f_1 \leq g_1$, $0\leq f_2 \leq g_2$,  and $f_1 + f_2 = f$.   
 Every element of $V_f(\Dr_{g_1, g_2})$ corresponds to a clutch of a curve $D$ whose isomorphism class lies in $V_{f_1}(\Bb_{E, g_1})$ and a curve $H$ from $V_{f_2}(\overline{\HH}_{g_2})$, identified at points $r \in D$ and $r' \in H$, the ramification points of the structural morphisms $D \to E$ and $H \to \Po$, for some choice of $(f_1, f_2)$ as above. (Note that there are only finitely many choices for $r$ and $r'$ as above given $D$ and $H$.) By the assumption, for every such a choice of $(f_1, f_2)$ it holds that $$\dim V_{f_1}(\Bb_{E, g_1}) + \dim V_{f_2}(\overline{\HH}_{g_2}) = (g_1 - 2 + f_1 - f_E) + (g_2 - 1 + f_2) = g - 3 + f - f_E, $$ and therefore, $\dim V_{f}(\Dr_{g_1, g_2}) = g - 3 + f - f_E$, for any $2\leq g_1 \leq g - 1$.

Now, we compute the dimension of $V_{f}(\Du_{g_1, g_2 - 1})$. Let $f_1$ and $f_2$ be numbers that satisfy $f_E\leq f_1 \leq g_1$, $0\leq f_2 \leq g_2 - 1$,  and $f_1 + f_2 = f - 1$. Let $D$ be a curve whose isomorphism class lies in $V_{f_1}(\Bb_{E, g_1})$, let $H$ be a hyperelliptic curve from $V_{f_2}(\overline{\HH}_{g_2 - 1})$, and $q_1 \in D$ and $q_1' \in H$ their points that are not ramification points of the structural morphisms $D \to E$ and $H \to \Po$. Denote by $q_2 \in D$ (resp. $q_2' \in H$) the image of $q_1$ (resp. $q_1'$) under the structural bielliptic (resp. hyperelliptic) involution. Every element of $V_f(\Du_{g_1, g_2 - 1})$ corresponds to a clutch of such curves $D$ and $H$ with $q_1$ and $q_1'$ being identified, as well as $q_2$ and $q_2'$. By the assumption we find that 
$$\dim V_{f_1}(\Bb_{E, g_1, 1}) + \dim V_{f_2}(\overline{\HH}_{g_2 - 1, 1}) = (g_1 - 2 + f_1 - f_E) + (g_2 - 2 + f_2) + 2 = g - 3 + f - f_E, $$ which implies $\dim V_{f}(\Du_{g_1, g_2 - 1}) = g - 3 + f - f_E$, for any $2\leq g_1 \leq g - 1$. 
\end{proof}

In particular, the preceding two lemmas imply that 
\begin{equation}
\dim (V_f(\Bb_{E, g}) \cap \Delta_0) \leq g - 3 + f - f_E, 
\label{eqn:intersection_with_delta0}
\end{equation}
for $g\geq 3$, under the assumption that $V_{f'}(\overline{\B}_{E, g'})$ is pure of dimension $g' - 2 + f' - f_E$ for all $2\leq g'<g$ and $f_E + 1\leq f' \leq g' - 1 + f_E$. 

Before proving our main result, let us introduce some notation we will use in the proof of the case $f_E = 0$. Let $j: \Mb_g \to \Tilde{\mathcal{A}}_g$ denote the Torelli morphism, where~$\Tilde{\mathcal{A}}_g$ is a fixed smooth toroidal compactification of the moduli space of principally polarized abelian varieties $\mathcal{A}_g$, as in \cite{alexeev}. We can now turn to the proof of our main result.

\begin{theorem}
\label{theorem:p_rank_stratification_MEg}
The locus $V_f(\overline{\B}_{E, g})$ is pure of dimension $g - 2 + f - f_E$, for any $g \geq 2$ and $f_E \leq f \leq g - 1 + f_E$. 
\end{theorem}

\begin{proof}
First, observe that $V_f(\overline{\B}_{E, g})$ is non-empty for any $f_E \leq f \leq g - 1 + f_E$. For $g = 2$, this follows from Lemma \ref{lemma:ME2}, and for $g \geq 3$, it follows from Lemmas \ref{lemma:glass_pries} and \ref{lemma:ME2}, which establish that $V_f(\Dr_{2, g - 2}) \subset V_f(\overline{\B}_{E, g})$ is non-empty. 

We follow the strategy of the proof of Theorem \cite[Theorem 2.3]{faber2004complete}. Let $W = W_f$ be a component of the $p$-rank $\leq f$ locus $V_f(\overline{\mathcal{B}}_{E, g})$ for a fixed $f$ with $f_E \leq f \leq g - 1 + f_E$. Let $W_{f + 1}$ be a component of $V_{f+ 1}(\Bb_{E, g})$ such that $W_f \subseteq W_{f + 1}$. This gives us a sequence $$W_f \subseteq W_{f + 1} \subseteq \ldots \subseteq W_{g - 2 + f_E} \subseteq W_{g - 1 + f_E}.$$ Note that $W_{g-1+f_E} = \overline{\B}_{E,g}$ and $\dim W_{g - 1 + f_E} = 2g - 3$ by Lemma~\ref{lem:Bbeq_is_proper_and_irreducible}. By Lemma \ref{Lem: purity we are using}, if $W_{i} \neq W_{i + 1}$, for $f \leq i \leq g - 2 + f_{E}$, then $W_{i}$ has codimension $1$ in~$W_{i + 1}$. This implies that the codimension of $W = W_f$ in~$\overline{\mathcal{B}}_{E,g} = W_{g - 1 + f_E}$ is at most $g - 1 + f_E - f$, or equivalently, that $$\dim W \geq g - 2 + f - f_E.$$

To establish an upper bound on the dimension of $W$, we use induction on $g$ with the base case $g = 2$ covered by Lemma \ref{lemma:ME2}. The inductive hypothesis assumes that for all $2\leq g' < g$ and $1\leq f' \leq f$, $V_{f'}(\overline{\M}_{E, g'})$ is pure of dimension $g' - 2 + f' - f_E$. Since $\B_{E, g}$ cannot contain complete subvarieties (see Lemma \ref{lem:Beq_is_quasi_affine_and_irreducible}) it follows that $W\cap \Delta \neq \varnothing$, where $\Delta = \Delta_0 \cup \Delta_1 \cup \ldots \cup \Delta_{\lfloor g/2 \rfloor}$ is the boundary divisor of $\Mb_g$. 

In the case $f_E = 1$, we find that the underlying curve of the generic point of $W$ is smooth by comparing dimensions. Therefore, $W \cap \Delta_i \neq \varnothing$ for some $0 \leq i \leq \lfloor \frac{g}{2} \rfloor$. In Lemmas \ref{lemma:prankbdry_res1} and \ref{lemma:prankbdry_res2}, we compute the dimensions of the loci of curves in $\Bb_{E, g} - \B_{E, g}$, under the assumption given by the inductive hypothesis, resulting in $\dim (W\cap \Delta_i) \leq g - 4 + f$ for any $0 \leq i \leq \lfloor \frac{g}{2} \rfloor$. Since $\Delta_i$ has codimension $1$ in $\overline{\mathcal{M}}_g$ and $\overline{\mathcal{M}}_g$  is a smooth stack, every component of $W\cap \Delta_i$ has codimension at most $1$ in $W$ (see \cite[page 614]{vistoli}), and we get the following inequality:
\begin{equation}
\dim W \leq \dim (W \cap \Delta_i) + 1 = g - 3 + f.
\label{eqn:dim_bdry}
\end{equation}
Consequently, we conclude that $\dim W \leq g - 3 + f = g - 2 + f - f_E$.

For $g\geq 3$ and $f_E = 0$, the situation requires more careful consideration. The inductive hypothesis gives $\dim V_f(\Duct_{1, g-2})  = g - 2 + f$, making it impossible to apply the previous argument directly. A stronger inequality than \eqref{eqn:dim_bdry} is needed. If $W$ consists entirely of isomorphism classes of singular curves, by comparing dimensions it has to be a component of 
$V_f(\Duct_{1, g-2})$, yielding $$\dim(W) = g-2+f=g-2+f-2f_E.$$
In the remaining case, the generic point of $W$ is represented by a smooth curve. That means that there is a component $\Gamma$ of $V_f(\B_{E, g}) \subseteq \M_g$ such that $W = \overline{\Gamma} \subseteq \Mb_g$. In particular, note that $\dim \Gamma = \dim W \geq g - 2 + f\geq 1$, and that $W \cap \Delta \neq \varnothing$ because of Lemma~\ref{lem:Beq_is_quasi_affine_and_irreducible}.
Moreover, since $\Gamma \subseteq \M_g$, it follows that $\dim \Gamma = \dim j(\Gamma)$, which consequently implies that $\dim W = \dim j(W)$, while $W \cap \Delta \neq \varnothing$ implies that $\dim j(W \cap \Delta)\geq 0$.
Let us prove the following inequality
\begin{equation}
\dim W \leq \dim j(W \cap \Delta) + 1,
\label{eqn:inequality_j_delta}
\end{equation}
which will give us the desired upper bound on $W$; we do this using a similar argument as in \cite[Section~1]{Oort1974}.
Assume that $\dim j(W) \geq \dim j(W\cap \Delta) + 2$. In a projective embedding of~$\Tilde{\mathcal{A}}_g$, we could then intersect $j(W)$ with $(\dim j(W) - 1)$-hyperplanes in general position to get a complete $1$-dimensional subvariety $X$ that does not meet $j(W \cap \Delta)$. This $X$ would be a subvariety of $j(W) - j(W\cap \Delta) = j(\Gamma)$, whose preimage in $\Mb_g$ would thus be a complete $1$-dimensional subvariety of $\Gamma$. This is impossible since~$\B_{E, g}$ and $\Gamma$ cannot contain such subvarieties by Lemma \ref{lem:Beq_is_quasi_affine_and_irreducible}. Therefore, the inequality \eqref{eqn:inequality_j_delta} holds. 
Note that $\dim j(V_f(\Duct_{1, g-2}) ) = g - 3 + f$ since the Torelli morphism "forgets" the pair $q_1, q_2 \in H$ of points in the same orbit of hyperelliptic involution as in \eqref{itm:prob} of the proof of Lemma \ref{lemma:prankbdry_res1}. Combining this with Lemmas \ref{lemma:prankbdry_res1} and \ref{lemma:prankbdry_res2}, we conclude that $\dim j(W \cap \Delta) \leq g - 3 + f$, which implies $\dim W \leq g - 2 + f$. This finishes the proof.
\end{proof}

\begin{remark}
Note that the arguments used in the proof of the preceding theorem for the case~$f_E = 0$ can be similarly applied to the case $f_E = 1$. However, we provide a separate proof for 
$f_E = 1$ to highlight the key differences arising from the nature of the elliptic curve $E$. 
\end{remark}

Below, we explicitly present some conclusions drawn from the proof of the preceding theorem and discuss their limitations.

\begin{corollary}\label{cor:pure of expected dim} Let $g \geq 2$. The following statements hold. 
\begin{enumerate}
    \item If $E$ is an ordinary elliptic curve then for any $1\leq f \leq g$, every generic point of $V_f(\overline{\B}_{E, g})$ is represented by a smooth curve. In particular, for any $g\geq 2$ and $1 \leq f\leq g$, there exists a smooth curve $D$ of genus $g_D = g$ and p-rank $f_C = f$, which is a double cover of $E$. 

    \item  If $E$ is a supersingular elliptic curve,  then for any $0\leq f \leq g - 1$, there exists a smooth curve $D$ of genus $g_D = g$ and p-rank $f_D = f$, which is a double cover of $E$, except in the case when $p = 3$, $g = 2$, and $f = 0$. 

    \item \label{itm:2}
    Let $E$ be a supersingular elliptic curve. Then, for any $0\leq f \leq g - 2$, there is a component of $V_f(\overline{\B}_{E, g})$ that consists entirely of singular curves.
\end{enumerate}
\label{cor:p-rank}
\end{corollary}
\begin{proof}
The third part follows from the proof of Theorem \ref{theorem:p_rank_stratification_MEg}. Moreover, we can explicitly describe these components. Given a supersingular elliptic curve $(E, O_E)$ in characteristic $p>0$, a component of $V_f(\overline{\mathcal{H}}_{g - 2})$, for $g\geq 4$, corresponding to a family of stable hyperelliptic curves $H$, and two distinct points $q, q' \in H$ in the same orbit under the hyperelliptic involution, a component of $V_f(\overline{\mathcal{B}}_{E,g})$ consists of isomorphism classes of stable curves of the form $(E, O_E)\cup_{q = O_E} (H, q, q') \cup_{q' = O_E} (E, O_E)$; note that the dimension of this family equals $\dim V_f(\overline{\mathcal{H}}_{g - 2, 1}) = g - 2 + f = g - 2 + f - f_E$, as desired. For $g = 2$, this follows from the proof of Lemma \ref{lemma:ME2}, while for $g = 3$, it follows similarly as in the proof of Lemma \ref{lemma:prankbdry_res1} using the identification $\overline{\mathcal{H}}_{1}\cong \overline{\mathcal{M}}_{1, 1}$.

To establish the existence of smooth bielliptic curves with a prescribed $p$-rank, we employ a technique similar to the one from the proof of \cite[Theorem~6.4]{pries_current_results}.
Let $p>3$ be a prime number and $E$ be an elliptic curve in characteristic $p$ whose $p$-rank equals $f_E$. By Lemma \ref{lemma:ME2}, for any $f_E\leq f'\leq 1 + f_E$, there is a smooth bielliptic curve~$D$ of genus $2$ with $f_D = f'$. By Lemma \ref{lemma:glass_pries} and by comparing dimensions, there is a smooth hyperelliptic curve $H \to \Po$ of genus $g - 2$ whose $p$-rank equals $f - f'$. Given ramification points $r \in D$ and $r' \in H$ of the structure bielliptic and hyperelliptic morphisms of $D$ and $H$, respectively, consider the curve $(D, r)\cup_{r = r'}(H, r')$. Any irreducible family of such curves has dimension $(2 - 2 + f' - f_E) + (g - 3 + f - f') = g - 3 + f - f_E$, and is therefore contained in a component of $V_f(\overline{\cB}_{E, g})$ that cannot consist entirely of isomorphism classes of singular curves, due to the description provided earlier and Theorem \ref{theorem:p_rank_stratification_MEg}. Consequently, a generic member of this component is represented by a smooth bielliptic curve of genus $g$ with $p$-rank $f$, which is a double cover of $E$. A similar argument applies for $p = 3$ and $f \neq 0$. In the remaining case, $p = 3$, $f_E = 0$, and $f = 0$, by Remark \ref{remark:ekedahl_p=3} there are no smooth curves of genus $2$ that are double covers of $E$. However, by Lemma \ref{p=3, genus 3, p rank 0} below, there is a smooth curve $D'$ of genus $3$ that is a double cover of $E$ and whose $3$-rank equals $0$. To show the existence of a smooth curve of genus $g\geq 4$ with $3$-rank $0$ that is a double cover of $E$, we now use the existence of $D'$ and a smooth (hyper)elliptic curve of genus $g - 3$ with $3$-rank $0$ and a similar argument to the one we just presented for $p>3$.
\end{proof}

\subsection{The $p$-rank stratification of the locus of bielliptic curves $\mathcal{B}_g$}\label{Sec: Bielleptic locus}

Recall that a smooth curve $D$ of genus $g \geq 2$ is bielliptic if there exists an elliptic curve $E$ and a double cover $D \to E$. The locus of bielliptic curves, $\mathcal{B}_g$, satisfies $\mathcal{B}_{E, g} \subset \mathcal{B}_g$ for any elliptic curve $E$, and $\dim \mathcal{B}_g = 2g - 2$.
To describe the closure of $\mathcal{B}_g$ in $\Mb_g$, we proceed similarly to the construction in Section \ref{section:construct_meg_stable}. Namely, one first considers the locus $\overline{\mathcal{B}}_g^{\rm adm}$ of admissible stable bielliptic curves. This is a proper Deligne-Mumford stack in characteristic $p \neq 2$ equipped with morphisms $\overline{\mathcal{B}}_g^{\rm adm} \to \Mb_{g, 2g - 2}$ and $\overline{\mathcal{B}}_g^{\rm adm} \to \Mb_{1, 2g - 2}$ which, respectively, extract the source curve and the target curve.  
By \cite[Theorem 3.7]{schmittvanzelm}, they induce well-defined proper morphisms $$\Psi: \overline{\mathcal{B}}_g^{\rm adm} \to \Mb_{g} \quad \text{ and } \quad \Phi: \overline{\mathcal{B}}_g^{\rm adm} \to \Mb_{1, 1}$$ obtained by forgetting the marked points and stabilizing. The locus of the admissible smooth bielliptic curves $\mathcal{B}_g^{\rm adm}$ is dense in $\overline{\mathcal{B}}_g^{\rm adm}$ and $\Psi: \mathcal{B}_g^{\rm adm} \to \mathcal{B}_g$ is a quasi-finite morphism obtained only by forgetting the admissible structure. It follows that $$\overline{\mathcal{B}}_g = \Psi(\overline{\mathcal{B}}_g^{\rm adm}) \subseteq \Mb_g,$$ the locus of stable bielliptic curves of genus $g$, is the closure of $\mathcal{B}_g$ in $\Mb_g$; see also Remark \ref{rmk: s-vz comparison}.

The description of the boundary of $\B_g$ inside $\Bb_g$ is analogous to that given in Section~\ref{section:boundary_description}. We highlight some details in the following remark.

\begin{remark}
\label{rmk:boundary_of_full_bielliptic}
By \cite[Section~3.4]{schmittvanzelm}, and similarly to our description in Section~\ref{section:boundary_description} and~\eqref{eqn:double_covers_bdry}, the boundary of $\B_g$ in $\Bb_g$ is the closure of the locus consisting of stabilizations of
double covers of
\begin{equation}
(E, O_E)\cup_{O_E = \infty} (\mathbb{P}^1, \infty) \quad \text{or} \quad E_0 = (\mathbb{P}^1, \infty, 0, 1)/\{0 = 1\},
\label{eqn:cor_double_covers_delta_0}
\end{equation}
ramified at some points, where $E$ is a smooth elliptic curve and $E_0$ is the singular elliptic curve, i.e., singular stable $1$-pointed genus-$1$ curve. 
    
    In particular, $\Bb_g \cap \Delta_0 = \Gamma_{ns}' \cup \Gamma_s'$, where:
\begin{itemize}
\item $\Gamma_{ns}'$ is the closure of the locus (of stabilizations) of double covers of $(E, O_E)\cup_{O_E = \infty} (\mathbb{P}^1, \infty)$ as in~\eqref{eqn:cor_double_covers_delta_0}, where $(E, O_E)$ is a smooth elliptic curve, and there is an even number of branch points of this cover on both $E - \{O_E\}$ and $\mathbb{P}^1 - \{\infty\}$. Set-theoretically, note that $\Gamma_{ns}'$ equals the union of $\Bb_{E, g} \cap \Delta_0$ over all isomorphism classes~$[E]$ of smooth elliptic curves~$E$. 

\item $\Gamma_s'$     
is the closure of the locus (of stabilizations) of double covers of $E_0$ as in~\eqref{eqn:cor_double_covers_delta_0}, where $E_0 = (\mathbb{P}^1, \infty, 0, 1)/\{0 = 1\}$ is the singular elliptic curve, branched at $2g - 2$ points in~${\mathbb{P}^1 - \{0, 1\}}$ (one of which is $\infty$) and at the node $0 = 1$. Note that every double cover $D \to E_0$ is of the form $D = (H, 0', 1')/\{0' = 1'\}$, where $H \to \mathbb{P}^1$ is a smooth hyperelliptic curve of genus~$g - 1$, and~$0', 1' \in H$ are ramification points above $0, 1 \in \mathbb{P}^1$. Consequently, $\Gamma_s'$ consists of isomorphism classes of $(H, 0', 1')/\{0' = 1'\}$, where $H$ is a stable hyperelliptic curve of genus~$g - 1$, and $0', 1'$ are two ramification points of the morphism $H \to C$, where $C$ is a rational curve.
\end{itemize}    
\end{remark}

Relying on Theorem~\ref{theorem:p_rank_stratification_MEg} above, we obtain the following result on the $p$-rank stratification of the full bielliptic locus~$\Bb_g$.
    
\begin{corollary}  For every $g \geq 2$ and $0\leq f \leq g$, the $p$-rank $\leq f$ locus of stable bielliptic curves $V_f(\overline{\mathcal{B}}_g)$ is pure of dimension $g - 2 + f$. 
\label{cor:biell_prank}
\end{corollary}

\begin{proof} Let $W$ be a component of $V_f(\overline{\mathcal{B}}_g)$. As before, we find that $\dim W \geq g - 2 + f$, using the purity of the $p$-rank stratification of $\Mb_g$. 
Consider the morphism $\Phi: \overline{\mathcal{B}}_g^{\rm adm} \to {\Mb}_{1, 2g - 2} \to {\Mb}_{1, 1}$ induced by $(D \to (E', P_1, \ldots, P_{2g - 2})) \mapsto ({E}, P_1)$, where $E = ({E}, P_1)$ denotes the elliptic curve obtained by stabilizing $(E', P_1)$. Let $W'$ denote the preimage of $W$ under the map $\Psi: \overline{\mathcal{B}}_g^{\rm adm} \to \Mb_g$. We separate the proof into different cases depending on whether the generic point of $W$ is represented by a smooth curve.
\begin{enumerate}
    \item Suppose that the generic point of $W$ is represented by a smooth curve; note that $\dim W = \dim W'$. There are two possibilities:
    \begin{enumerate}
    \item $\Phi(W')$ is a single point in $\M_{1,1} \subseteq \Mb_{1,1}$ corresponding to a smooth elliptic curve~$E$. In this case, $W$ is an irreducible locus consisting of isomorphism classes of smooth curves~$D$ with $f_D \leq f$, each admitting a double cover $D \to E$.
In other words, $W \subseteq V_f(\overline{\cB}_{E, g})$. Therefore, $$\dim W \leq \dim V_f(\overline{\cB}_{E, g}) = g - 2 + f - f_E \leq g - 2 + f$$ by Theorem~\ref{theorem:p_rank_stratification_MEg}. 

    \item Otherwise, $\Phi(W')$ consists of at least two points in $\Mb_{1,1}$. Moreover, $\Phi(W') \subseteq \Mb_{1,1}$ is irreducible, being the image of an irreducible locus under a morphism. These two properties imply that $\Phi(W') = \Mb_{1,1}$. In particular, there exists an element of~$W'$ that maps to the singular elliptic curve $E_0 = (\mathbb{P}^1, \infty, 0, 1)/\{0 = 1\}$ under~$\Phi$, and hence whose image under~$\Psi$ lies in~$W \cap \Delta_0$. Therefore, $W \cap \Delta_0 \neq \varnothing$.
Since $\Mb_g$ is a smooth stack, we have $\dim W \leq \dim (W \cap \Delta_0) + 1$. Thus, to conclude that $\dim W \leq g - 2 + f$, it suffices to show that $\dim (W \cap \Delta_0) \leq g - 3 + f$.

   Let $\Gamma_{ns} = \Gamma_{ns}' \cap W$ and $\Gamma_s = \Gamma_s' \cap W$, where $\Gamma_{ns}'$ and~$\Gamma_s'$ are the loci described in Remark~\ref{rmk:boundary_of_full_bielliptic}.
It suffices to show that $\dim \Gamma_{ns} \leq g - 3 + f$ and $\dim \Gamma_s \leq g - 3 + f$.
Recall that $W \subseteq V_f(\Bb_g)$.
    \begin{itemize}
        \item To show $\dim \Gamma_{ns} \leq g-3+f$, note that $\Gamma_{ns}$ is the union of the loci $V_f(\Bb_{E, g}) \cap \Delta_0$ over all isomorphism classes $[E]$ of (smooth) elliptic curves $E$.
        By Theorem~\ref{theorem:p_rank_stratification_MEg} and Lemmas~\ref{lemma:prankbdry_res1} and~\ref{lemma:prankbdry_res2}, no component of $V_f(\Bb_{E, g})$ is contained in $\Delta_0$, and
        $\dim (V_f(\Bb_{E, g}) \cap \Delta_0) \leq g - 3 + f - f_E$.
        Moreover, the locus of curves in $\M_{1,1}$ with $p$-rank $f_E$ has dimension $f_E$ for $0 \leq f_E \leq 1$.
        Hence, we conclude that
        $$\dim \Gamma_{ns} \leq (g - 3 + f - f_E) + f_E = g - 3 + f.$$

        \item  To show $\dim \Gamma_{s} \leq g-3+f$, let $V_{f - 1}(\overline{\mathcal{H}}_{g - 1, 2})$ be the locus of stable $2$-pointed hyperelliptic curves $(H, q_1, q_2)$ of genus $g - 1$ whose $p$-rank is $\leq f - 1$, and let $\Gamma'' \subseteq V_{f - 1}(\overline{\mathcal{H}}_{g - 1, 2})$ be the sublocus where the marked points $q_1, q_2$ of $(H, q_1, q_2)$ are chosen to be two ramification points of the double cover $H \to C$, with $C$ a rational curve. Since there are only finitely many ramification points, we have $\dim \Gamma'' = \dim V_{f - 1}(\overline{\mathcal{H}}_{g - 1})$. Using the definition of $\Gamma_s'$ and the inclusion $W \subseteq V_f(\overline{\mathcal{B}}_{g})$, we deduce that $\Gamma_s = \Gamma_s' \cap W$ is contained in the image of $\Gamma''$ under the clutching morphism $\overline{\mathcal{M}}_{g - 1, 2} \to \overline{\mathcal{M}}_{g}$. Lemma~\ref{lemma:glass_pries} then implies that $$\dim \Gamma_s = \dim V_f(\overline{\mathcal{H}}_{g - 1}) \leq g - 3 + f.$$  
    \end{itemize}

    \end{enumerate}
    \item  If $W$ is not generically represented by a smooth curve, then $W$ consists entirely of isomorphism classes of singular curves.  
    If the image of $W'$ under $\Phi$ contains the isomorphism class of $E_0 = (\mathbb{P}^1, \infty, 0, 1)/\{0 = 1\}$, the same discussion as above implies that $W \cap \Delta_0 \neq \varnothing$ and $\dim W \leq g - 2 + f$; this covers the cases $\Phi(W') = \overline{\mathcal{M}}_{1,1}$ or $\Phi(W') = [E_0]$.
    Otherwise, $\Phi(W') = [E]$, where $E = (E, O_E)$ is a smooth elliptic curve. As before, Theorem~\ref{theorem:p_rank_stratification_MEg} implies that $\dim W \leq \dim V_f(\overline{\mathcal{B}}_{E, g}) = g - 2 + f - f_E \leq g - 2 + f$. (More precisely, $\dim W \geq g - 2 + f$, and Lemmas \ref{lemma:prankbdry_res1} and \ref{lemma:prankbdry_res2} imply that $f_E = 0$ in this case and that $W$ could only possibly be the closure of the locus consisting of isomorphism classes of double covers of curves of the form $(E, O_E) \cup_{O_E = \infty} (\mathbb{P}^1, \infty)$, ramified above $2g - 2$ distinct points $P_i \in \mathbb{P}^1 - \{\infty\}$, $1 \leq i \leq 2g - 2$, and unramified otherwise, as in Corollary \ref{cor:p-rank}.)

\end{enumerate} 
In both cases, we conclude that $\dim W \leq g - 2 + f$, which implies $\dim W = g - 2 + f$, as desired. Consequently, $V_f(\Bb_g)$ is pure of dimension $g - 2 + f$.
\end{proof}

Theorem \ref{thm:mainintro} and Corollary \ref{cor:p-rankintro} follow from Theorem \ref{theorem:p_rank_stratification_MEg}, and Corollaries \ref{cor:p-rank} and~\ref{cor:biell_prank}.

\section{Supersingular double covers of $E$}\label{sec:applications}
In this section, we look more carefully into the supersingular locus of $\B_{E,g}$ and offer some observations as a consequence of the results we establish in Section \ref{Sec: p rank strat MEg}.
Let $k$ be an algebraically closed field of characteristic $p>2$, as above. Recall that a smooth curve $D$ of genus $g$ over $k$ is called supersingular if its Jacobian $\Jac(D)$ is isogenous to a product $E^g$, where~$E$ is one of the finitely many  (up to isomorphism) supersingular elliptic curves over $k$. Equivalently, its Newton polygon, the combinatorial way to encode the isogeny classes of $\Jac(D)[p^{\infty}]$, is of the form $\left ( {\frac{1}{2}, \frac{1}{2}, \ldots, \frac{1}{2}} \right )$. In particular, if $D$ is supersingular, it has $p$-rank $0$. We denote by $\mathcal{M}_g^{\rm ss} \subseteq V_0(\mathcal{M}_g)$ the locus of all supersingular smooth curves of genus $g$ in characteristic $p$. 
\\ 

We show some applications of Theorem \ref{theorem:p_rank_stratification_MEg} in low genera $g$, for any $p>2$:
\begin{enumerate}
    \item[($g = 2$)] \textit{There are only finitely many supersingular curves of genus $2$ that are double covers of an elliptic curve (necessarily a supersingular one)}. This follows from the fact that $V_0(\B_{E, 2})$ is a $0$-dimensional locus and the fact that for $g = 2$, being supersingular is equivalent to having $p$-rank $0$. See Remark~\ref{remark:superspecial_genus2} for more details. 
    
    \item[($g = 3$)] \textit{The locus of supersingular curves of genus $3$ that are double covers of an elliptic curve is $1$-dimensional}. Indeed, we have shown that $V_0(\B_{E, 3})$ is a $1$-dimensional locus consisting entirely of supersingular genus 3 curves. Specifically, the classification of eligible Newton polygons of height $2g = 6$ indicates that any genus-$3$ curve with $p$-rank $0$ and a map to an elliptic curve is automatically supersingular. See also \cite[Theorem 1.12, Remark 3.3]{oort_hess}.
    
    \item[($g = 4$)] \textit{The generic point of any component of the supersingular locus $\mathcal{M}_4^{\rm ss}$ of $\mathcal{M}_4$ does not correspond to a smooth curve that is a double cover of an elliptic curve}. Kudo-Harashita-Senda in \cite{Kudo2020} or Pries in \cite{Pries2023} show the existence of a smooth curve of genus $4$ which is supersingular in any characteristic $p>0$, so that $\mathcal{M}_4^{\rm ss}$ is non-empty; any of its components is either $3$-dimensional or $4$-dimensional. On the other hand, we have shown that for any supersingular elliptic curve~$E$, $V_0(\B_{E, 4})$ is pure of dimension $2$, which implies the result.
    
    Note that there are only $2$ possible Newton polygons for a curve $D$ of genus $4$ with $p$-rank~$0$ which is a double cover of an elliptic curve. Namely, the Newton polygon of $D$ is either  $\left(\frac{1}{3}, \frac{1}{3}, \frac{1}{3}, \frac{1}{2}, \frac{1}{2}, \frac{2}{3}, \frac{2}{3}, \frac{2}{3}\right)$ or the supersingular one $\left(\frac{1}{2}, \frac{1}{2}, \frac{1}{2}, \frac{1}{2}, \frac{1}{2}, \frac{1}{2}, \frac{1}{2}, \frac{1}{2}\right)$.

    Moreover, for $p>3$, the supersingular curve that  Kudo-Harashita-Senda found is a Howe curve $H$, i.e., a normalization of the fiber product of two elliptic curves $E_i, i = 1, 2$, such that the degree $2$ maps $E_i \to \mathbb{P}^1$ share exactly one branch point. In particular, $H$ is a double cover of $E_1$ and $E_2$. If we take now $E = E_i$ for any $i = 1, 2$, this shows that the locus $\B_{E, 4}^{\rm ss}$ of supersingular curves in $\B_{E, 4}$ has dimension $$1\leq \dim \B_{E, 4}^{\rm ss}\leq 2,$$ because of the purity theorem \cite[Theorem 4.1]{de2000purity}. 

\end{enumerate}

\textbf{Question.} Given $p>2$ and $g\geq 5$, is there a smooth supersingular curve of genus $g$ which is a double cover of an elliptic curve in characteristic $p$?

\begin{example}\label{ss double cover for g eq five}
    When $g=5$, we consider the fiber product of an elliptic curve and a hyperelliptic curve of $g=2$. We fix $p \equiv 11 \pmod{12}$.

    Let
    \[E': y^2=(x^2+1)(x^2-1)\]
    \[D': z^2=(x^4+x^2+1)(x^2-1)\]
    By Lemma \ref{lem: p rank 0 construction}, we see that when $p \equiv 11 \pmod{12}, E', D'$ both have $p$-rank $0$. Let $D''$ be the normalization of the fiber product of $E', D'$ over $\Po$. Then the $p$-rank of $D''$ is the $p$-rank of $w^2=(x^2+1)(x^4+x^2+1)$. Following the strategy of \cite{Kudo2020}, we parametrize curves isomorphic to $E',D'$ and solve for the parameters that give fiber product $p$-rank zero.

    Let
    \[E_{u}: y_1^2=((x+u)^2+(ux+1)^2)((x+u)^2-(ux+1)^2)\]
    \[D_v: z_1^2=((x+v)^4+(x+v)^2+1)((x+v)^2-(vx+1)^2)\]
    \[D_{u,v}: w_1^2=(1-u^2)(1-v^2)((x+u)^2+(ux+1)^2)((x+v)^4+(x+v)^2+1)\]
    Note that $E_u \cong E'$, $x=\frac{x_1+u}{ux_1+1}, y=\frac{y_1}{(ux_1+1)^2}$. Observe also $D_v \cong D'$, with $x=\frac{x_1+v}{vx_1+1}, y=\frac{y_1}{(vx_1+1)^3}$. Then let $f(x)=(x+u)^2+(ux+1)^2)((x+v)^4+(x+v)^2+1)$. The Hasse-Witt matrix of $D_{u,v}$ is a~$2 \times 2$ matrix~$M$ with $M_{i,j}=c_{pj-i}$, where $c_{pj-i}$ is the coefficient of $x^{pj-i}$ in $f(x)^{\frac{p-1}{2}}$. Then, $D_{u,v}$ is supersingular if and only if $M M^{(p)}=0$. By \cite[proposition 2.2]{Kudo2020}, this  is equivalent to 
    \begin{equation}\label{eq: p rank zero eq}
       ad - bc = 0, \hspace{5mm} ab^{p-1} + d^{p} = 0, \hspace{5mm} a^p + c^{p-1}d = 0 
    \end{equation}
    where
    \[a = c_{p-1}, \hspace{5mm} b = c_{2p-1}, \hspace{5mm} c = c_{p-2}, \hspace{5mm} d = c_{2p-2}\]
    Since $a,b,c,d$ are polynomials in terms of $u,v$, we plug in $u,v $ into the expression of $a,b,c,d$. Then, the three equations in \eqref{eq: p rank zero eq} become $h_i(u,v)$ for $1 \leq i \leq 3$. 
    
    Therefore, the $(u,v)$ such that $E_u \times_{\Po} D_v$ is supersingular are the solutions of $h_i(u,v)=0$ for $1 \leq i \leq 3$. Also we need to exclude $u=\pm 1, v=\pm 1$ and the $(u,v)$ such that $((x+u)^2+(ux+1)^2)$ and $((x+v)^4+(x+v)^2+1)$ have a common root. We can find a root for $(u,v)$ in $\mathbb{F}_p$ for $p<1000$, except $p=107, 443, 491$. So if we assume $p \neq 107, 443, 491, p \leq 1000, p \equiv 11\pmod{12}$, we can construct $D=E_u \times_{\Po} D_v$ such that $f_D=0$. By the above analysis, this implies that $D$ is a double cover of~$E_u$ that is supersingular.  
\end{example}

We conclude with the following observations.
\begin{enumerate}
    \item[($g = 5$)] We have shown that $V_0(\B_{E, 5})$ is pure of dimension $3$. If the supersingular locus of $\M_5$ is non-empty, the expected dimension $d_e$ of any of its components $W$ is $d_e = 3$, but a priori, the dimension of $W$ could be any number $d$ such that $3\leq d \leq 6 = \dim (\mathcal{S}_5)$, where $\mathcal{S}_5 \subseteq \mathcal{A}_5$ is the locus of supersingular principally polarized abelian varieties.
    
    Similarly to the cases $(g = 3)$ and $(g = 4)$, we can see that there are only $3$ possible Newton polygons for a curve $D$ of genus $5$ with $p$-rank $0$ which is a double cover of an elliptic curve. Namely, the Newton polygon of $D$ is one of the following: $ (\frac{1}{4}, \frac{1}{4}, \frac{1}{4}, \frac{1}{4}, \frac{1}{2}, \frac{1}{2}, \frac{3}{4}, \frac{3}{4}, \frac{3}{4}, \frac{3}{4})$, $(\frac{1}{3}, \frac{1}{3}, \frac{1}{3}, \frac{1}{2}, \frac{1}{2}, \frac{1}{2}, \frac{1}{2}, \frac{2}{3}, \frac{2}{3}, \frac{1}{3})$, or
  $(\frac{1}{2}, \frac{1}{2}, \frac{1}{2}, \frac{1}{2}, \frac{1}{2}, \frac{1}{2}, \frac{1}{2}, \frac{1}{2}, \frac{1}{2}, \frac{1}{2})$, the supersingular one.

    Therefore, for any $E$ such that there is a supersingular double cover $D \to E$ of genus $5$, it follows that $$1\leq \dim \B_{E, 5}^{\rm ss} \leq 3.$$ In particular, for $E$ as in Example \ref{ss double cover for g eq five}, we find that there is at least a $1$-dimensional family of smooth curves $D$, which are double covers of $E$.
\end{enumerate}

\textbf{Question.} Given $p>2$, is there a ($3$-dimensional) component of $V_0(\B_{E, 5})$ that entirely consists of supersingular curves in characteristic $p$? If so, is this a component of $\M_5^{\rm ss}$, the supersingular locus of $\M_5$? Such a component of $\M_5^{\rm ss}$ would have a non-trivial generic automorphism group. It is currently not known whether there exists any $(p,g)$ with $g \geq 5$ such that $\mathcal{M}_g^{\rm ss}$ has a component whose generic point has a non-trivial automorphism group in characteristic $p$.

\section{Explicit families of smooth covers of small $p$-rank} \label{sec: explicit}

In Theorem~\ref{theorem:p_rank_stratification_MEg}, we showed that $V_f(\Bb_{E,g})$ is pure of dimension $g - 2 + f - f_E$. In this section, we explicitly construct curves and one-parameter families $\D_t$ inside $V_f(\B_{E,g})$ for $f$ around $\frac{g}{2}$. These curves and families are obtained as the normalization of fiber products of the fixed elliptic curve~$E$ and certain hyperelliptic curves $H$. By imposing conditions on the Hasse-Witt matrix of $H$, we obtain double covers of $E$ whose $p$-rank is much smaller than the generic $p$-rank.

Using this technique, we first construct a genus-$g$ double cover of a supersingular elliptic curve~$E$ in characteristic $p = 3$ for $g = 3$. (Recall that this construction provides the base case for Corollary~\ref{cor:pure of expected dim} when $p = 3$.)

\begin{proposition}[$p = 3$]\label{p=3, genus 3, p rank 0}
Let $E$ an elliptic curve whose $3$-rank equals $0$, i.e., with $f_E = 0$. Then, there exists a smooth curve $D$ of genus $3$ which is a double cover of $E$ and has $f_D = 0$.
\end{proposition}

For more general values of $p$ and $g$, we use this technique to construct one-parameter families $\D_t \subseteq \B_{E,g}$, where each curve $D \in \D_t$ is a smooth genus-$g$ double cover of $E$ with $p$-rank $f_D$ much smaller than $g$. In particular, we construct the following in Proposition~\ref{prop: general $g$}:

\begin{enumerate}
    \item Let $p>2$ and $g>3$, and suppose $g-1 \nmid p-1$ and $p\nmid g-1$. Then, we construct a one parameter family $\mathcal{D}_t$ such that each curve $D \in \D_t$ is a genus-$g$ double cover of $E$ with $f_D \leq g-1$.
    \item With the same assumption as above, if we assume furthermore that $g-1 \mid p+1$, then we construct $\D_t$ such that each $D \in \D_t$ is a genus-$g$ double cover of $E$ with $f_D \leq \lceil\frac{g}{2}\rceil+f_E$. 
\end{enumerate}

We can strengthen the upper bound by imposing additional constraints on the congruence class of $p$. In particular, in Propositions~\ref{prop: odd n} and~\ref{prop: even n}, we construct the following cases:
\begin{enumerate}
    \item Let $p>2$, $g>3$ and $g$ being even. Suppose $g+2 \mid p+1$ and $p \nmid g+2$. Then, we can construct the family $\D_t$ such that $D \in \D_t$ is a genus-$g$ double cover of $E$ with $f_D \leq \lceil\frac{g}{2}\rceil+f_E-1$.
    \item Let $p>2$, $g>3$ and $g$ being odd. Suppose $g+1 \mid p+1$ and $p \nmid g+1$. Then, we can construct the family $\D_t$ such that $D \in \D_t$ is a genus-$g$ double cover of $E$ with $f_D \leq \lceil\frac{g}{2}\rceil+f_E-1$.
\end{enumerate}
\subsection{Decomposition of the Jacobian}
Here we describe the construction based on fiber products. Let $k$ be an algebraically closed field of characteristic $p > 2$, and let $E: y^2 = x(x-1)(x-\lambda)$, for some $\lambda \in k$, be an elliptic curve over $k$ in Legendre form.

Recall that the $p$-rank of a curve $D$ of genus $g$ over $k$ is the semisimple rank of the Frobenius morphism on $H^1(D, \OO_D)$; that is, the rank over $k$ of the $g^{\text{th}}$ semilinear iteration of the Hasse-Witt matrix associated to $D$. Since the entries of the Hasse-Witt matrix are polynomials in the coefficients of $D$, we aim to construct $D \to E$ with a prescribed branching locus. One way to achieve this is to project the branching locus onto $\mathbb{P}^1$, find a hyperelliptic curve $H$ branched at the projection, and take the fiber product of $E$ and $H$.

We take the branching locus of the cover $D \to E$ to consist of $\Z/2\Z$-conjugates, i.e., its branching points should be of the form $$(x_1,y_1), (x_1,-y_1), \cdots, (x_{g-1},y_{g-1}), (x_{g-1},-y_{g-1})$$ 
Where $x_1, \dots, x_k \in \Po(k)$ and $x_k \not \in \{0,1,\lambda\}$. After fixing the branching points, we consider the fiber product $D' = H \times_{\Po} E$ where 
\begin{equation}\label{Eq: C' equation}
\begin{split}
& E: y^2=x(x-1)(x-\lambda), \\
& H: z^2=\prod_{k=1}^{g-1}(x-x_k). \\
\end{split}
\end{equation}

We then take $D$ to be the normalization of $D'$. The resulting map $D \to E$ is a degree~$2$ cover, branched along the locus $\{(x_1,y_1), (x_1,-y_1), \cdots, (x_{g-1},y_{g-1}), (x_{g-1},-y_{g-1})\}$. 

We now aim to understand the Hasse-Witt matrix of $D$. To do this, we decompose the Jacobian variety $\Jac(D)$, up to isogeny, into simpler components that can be analyzed individually. We make the following observation.
\\
\textbf{Observation}\label{observation: decomposition}. 
    If $D$ is as above, then:
    \begin{itemize}
        \item The automorphism group of $D$ contains $\Z/2\Z \times \Z/2\Z$. It is  generated by the order two elements $\iota_1: (x,y,z) \to (x,y,-z)$ and $\iota_2: (x,y,z) \to (x,-y,z)$. 
        \item Let $\iota_3=\iota_1 \circ \iota_2$. Consider the order two subgroup groups of $\Aut(D)$ given by $G_1=\{id,\iota_1\}$, $G_2=\{id,\iota_2\}$ and $G_3=\{id,\iota_3\}$. Then we get the following three curves as quotients of $D$:
\begin{enumerate}
\item $E=D/G_1: y^2=x(x-1)(x-\lambda)$.
This is the elliptic curve $E$;
\item $H=D/G_2: z^2=\prod_{k=1}^{g-1}(x-x_k)$. This is the hyperelliptic curve $H$ in the construction, and its genus is  $g_{H}=\lceil\frac{g-1}{2}\rceil -1$;
\item $H'=D/G_3: w^2 = x(x-1)(x-\lambda)\prod_{k=1}^{g-1}(x-x_k)$. This is a new hyperelliptic curve of genus $g_{H'}=\lceil \frac{g}{2}\rceil$. Here $w=yz$.
\end{enumerate}
\end{itemize}
Then, the Jacobian of $D$ decomposes, up to isogeny, as 
\begin{equation}\label{eq: decomposition of Jacobian}
    \Jac(D) \sim \Jac(E) \oplus \Jac(H) \oplus \Jac(H')
\end{equation}
If we let $\Jac(D)_{\textup{new}} \subset \Jac(D)$ denote the Prym variety of the cover $D \to E$, where $\Jac(D)_{\textup{new}}$ consists of divisors that are negated by the non-trivial element of $\Aut(D/E)$, then we also have $\Jac(D)_{\textup{new}} \sim \Jac(H) \oplus \Jac(H')$. 


\begin{proof}
    These observations follow immediately from the construction in~\eqref{Eq: C' equation}, combined with \cite[Theorem C]{KaniRosen1989} and the fact that $\Jac(D) \sim \Jac(D)_{\textup{new}} \oplus \Jac(E)$.
\end{proof}

We can now use the fiber product to construct double covers of $E$ of genus $g = 3$ with $f_D = 0$ in characteristic $p = 3$. By Remark~\ref{remark:ekedahl_p=3}, there are no smooth curves of genus $g = 2$ that are double covers of such an $E$.

\begin{proof}[Proof of Proposition \ref{p=3, genus 3, p rank 0}]
We use the fact that in characteristic $3$, $y^2=x^4+ax^3+bx^2+cx+d$ is a supersingular elliptic curve if and only if $b=0$ and $x^4+ax^3+cx+d=0$ has $4$ distinct roots. This is because we need $4$ distinct roots to ensure that the curve is smooth and then the Hasse-invariant is $b$.

We write $i\in \mathbb{F}_9$ for any of the roots of $x^2=-1$ and consider
\[E:y^2=x(x-i-2)(x-2i-2)(x-1)=x^4+x^3+x.\]
Note that over an algebraically closed field of characteristic $3$, there is up to isomorphism a unique supersingular elliptic curve, so, without loss of generality, we can assume $E$ is given by the above equation.
Let $D$ be the normalization of $H \times_{\Po} E$, where 
\begin{align*}
H: z^2&=x(x-i-2)(x-2i)(x-i-1)=x^4+2ix^3+ix.
\end{align*}
From \ref{eq: decomposition of Jacobian}, we see that $\Jac(D) \sim \Jac(E) \oplus \Jac(H) \oplus \Jac(H')$, where
\begin{align*}
&H: &z^2=x(x-i-2)(x-2i)(x-i-1)&=x^4+2ix^3+ix,\\
&H': &w^2=(x-2i-2)(x-1)(x-2i)(x-i-1)&=x^4+(i+2)x^3+(2i+2)x+1.
\end{align*}
We see that $E$, $H$ and $H'$ are supersingular and $g_{D}=3$. Thus $D$ is a genus-$3$ double cover of $E$ and the $3$-rank of $D$ is $0$.
\end{proof}

\subsection{Explicit one-parameter families in $V_f(\B_{E,g})$ with $f \approx \frac{g}{2}$}

As we see from \eqref{eq: decomposition of Jacobian}, to construct $D$ with smaller $p$-rank, it suffices to find hyperelliptic curves $H$ with small $f_H$.
The authors of \cite{devalapurkar2017dieudonn} considered curves of the form $z^2 = x^n + 1$ under the condition that both $n$ and $(n-1)/2$ are prime. Furthermore, assuming $p \not\equiv 0,1 \pmod{n}$, they showed that such curves have $p$-rank zero. We consider curves of a similar form, without requiring $n$ or $(n-1)/2$ to be prime, and determine the conditions under which this family has small $p$-rank.


\begin{lemma}\label{lem: p rank 0 construction}
Let $n \geq 3$ and $t \in k^* = k - \{0\}$, and let $H$ denote the curve given by $H:z^2=x^n-t^n$. If $p\nmid n$ and $p\not\equiv 1\pmod{n}$, then $H$ is not ordinary, i.e.,  $f_H \leq g_H-1$. Moreover, if $p\equiv -1\pmod{n}$, then $H$ is superspecial, so in particular $f_H=0$.
\end{lemma}
\begin{proof}
For $H:z^2=x^n-t^n$, its genus is $g_H=\lceil\frac{n}{2}\rceil-1$. Let $M$ be the Hasse-Witt matrix of $H$. Note that $H$ is ordinary if and only if the $p$-rank of $D$ is $g_H$, which is if and only if the rank of $M$ is $g_H$.

It is known that $H^1(H,\mathcal{O}_{H})$ has basis given by $\frac{z}{x^k}$ for $1\leq k\leq g_H$.
Note that $\Z/n\Z$ acts on $H$, where $1 \in \Z/n\Z$ acts by sending $(x,z)$ to $(\zeta_n x,z)$. Since $p \nmid n$, this action leads to the decomposition $$H^1(H,\mathcal{O}_{H})=\bigoplus_{i=0}^{n-1}H^1(H,\mathcal{O}_{H})_i.$$ 
Here, $H^1(H,\mathcal{O}_{H})_i$ is the eigenspace where $1 \in \Z/n\Z$ acts via scalar multiplication by $\zeta_n^i$. 


By the computation in \cite[\S 5]{Bouw_2001}, 

$$H^1(H,\mathcal{O}_{H})_i=\begin{cases}
    \left\langle \frac{z}{x^{n-i}} \right\rangle & \text{if $ n-g_H\leq i\leq n-1$} \\
    0 & \text{if $0\leq i\leq n-g_H-1$}
\end{cases}$$ 
Moreover, $M$ maps $H^1(H,\mathcal{O}_{H})_i$ to $H^1(H,\mathcal{O}_{H})_{pi\pmod{n}}$. This means the rank of~$M$ equals $\#S_1$, where $S_1$ is the following set:
\[S_1  =\{n-g_H\leq i\leq n-1 \mid n-g_H\leq pi\pmod{n}\leq n-1\}.\]
Clearly, it has the same size as
\[S_2=\{1\leq i\leq g_H \mid 1\leq pi\pmod{n}\leq g_H\}.\]
Note that $H$ is ordinary if and only if $S_2=\{1,2,\dots,g_H\}$. 
Let $a$ be the congruence class of $p$ modulo $n$. The assumption $p \not \equiv 1 \pmod{n}$ implies $2\leq a\leq n-1$. Depending on the value of $a$, we have the following cases:
\begin{enumerate}
    \item If $g_H+1\leq a\leq n-1$: then $1\notin S_2$, so $H$ is not ordinary.
    \item If $2\leq a\leq g_H$: Let $i=\lceil\frac{g_H+1}{a}\rceil$, so $1\leq i\leq g_H$. Moreover, observe $$g_H+1\leq ai< g_H+1+a\leq 2g_H+1\leq n.$$ This implies that $i\notin S_2$, so $H$ is not ordinary.
\end{enumerate}
In both cases, $H$ is not ordinary.
In addition, if $p \equiv -1 \pmod{n}$, then $S_2=\varnothing$. Therefore, $M$ has rank $0$ and hence the $p$-rank of $H$ is also $0$.
\end{proof}

\begin{remark}
For the curves $H: y^2 = x^n - t^n$ considered in Lemma~\ref{lem: p rank 0 construction}, it is known that $\Jac(H)$ decomposes as a product of CM abelian varieties. Therefore, the Newton polygon of $H$ can be computed using the Shimura-Taniyama formula \cite[Section~5]{Tate}. Since the Newton polygon determines the $p$-rank, this approach provides an alternative, though less direct, proof of Lemma~\ref{lem: p rank 0 construction}.
\end{remark}

Now, we can can construct $D$ with small $p$-rank by carefully choosing the branching locus of $H$.

\begin{proposition}\label{prop: general $g$}
Let $g>3$, $n=g-1$, and let $E: y^2=x(x-1)(x-\lambda)$ be an elliptic curve over~$k$. 
For $t \in k^*$, let $H_t$ denote the hyperelliptic curve $H_t: z^2 = x^n - t^n$, and let $D_t$ denote the normalization of $H_t \times_{\mathbb{P}^1} E$. Then we have the following:
\begin{enumerate}
\item Suppose $p\not\equiv 1\pmod{n}$ and $p \nmid n$. Then $f_{D_t} \leq f_E+g-2$. In particular, $D_t$ is not ordinary.

\item Suppose further $p \equiv -1 \pmod{n}$, then we have $f_{D_t} \leq f_E+\lceil\frac{g}{2}\rceil$. In particular, it is bounded from above by $\lceil\frac{g}{2}\rceil+1$.
\end{enumerate}

\end{proposition}
\begin{proof}
    We first show that $\Jac(D_t)_{\textup{new}}$ is not ordinary. Observe that $D_t$ has two hyperelliptic quotients: \[H_t: z^2=x^n-t^n  \text{ and }  H'_t: w^2=x(x-1)(x-\lambda)(x^n-t^n)\]
    Since $\Jac(D_t)_{\textup{new}} \sim \Jac(H_t) \oplus \Jac(H'_t)$, it suffices to show that at least one of the factors is not ordinary. By Lemma \ref{lem: p rank 0 construction}, we have that $f_{H_t} \leq g_{H_t}-1$, so $H_t$ is not ordinary, and hence $\Jac(D)_{\textup{new}}$ is not ordinary. Furthermore, since $\Jac(D_t)\sim \Jac(D_t)_{\textup{new}} \oplus \Jac(E)$, we see that $f_{D_t} \leq f_E+g-2$. Moreover, when $p \equiv -1 \pmod{n}$, we obtain $f_{H_t}=0$, so $f_{\Jac(D_t)_{\textup{new}}} = f_{H_t'} \leq g_{H_t'}=\lceil \frac{n+1}{2} \rceil=\lceil \frac{g}{2} \rceil$. Hence, $f_{D_t} \leq \lceil \frac{g}{2} \rceil+f_E \leq \lceil \frac{g}{2} \rceil+1$.


\end{proof}

Recall the decomposition in~\eqref{eq: decomposition of Jacobian}. In Proposition~\ref{prop: general $g$}, we force $D$ to have a small $p$-rank by arranging that $H$ has small $p$-rank. In Propositions~\ref{prop: odd n} and~\ref{prop: even n}, we force $D$ to have an even smaller $p$-rank by arranging that $H'$ has small $p$-rank, noting that $g_{H'}>g_H$.
Note that, when $g$ is even, the upper bound of $p$-rank in Proposition \ref{prop: general $g$} is $\frac{g}{2}$+1. In the following proposition, we construct a family $\D_t$ with upper bound being $\frac{g}{2}$ given that $g$ is even.
\begin{proposition}\label{prop: odd n}
     Let $g>3$ be an even number and $p \equiv -1 \pmod{g+2}$. Fix the elliptic curve ${E: y^2=x(x-1)(x-\lambda)}$ over $k$ and a primitive $(g+2)^{th}$ root of unity $\zeta_{g+2}$. For $t \in k^*$, let $H_t$ denote the hyperelliptic curve $H_t: z^2=\prod_{k=1}^{g-1}(x-x_k)$, where  the branching points $x_1, \cdots, x_{g-1} \in \Po(k)$ are chosen, such that, 
     if $T$ is the unique linear fractional transformation that sends $(t, \zeta_{g+2}t, \zeta_{g+2}^2t)$ to~$(0,1, \lambda)$, then $x_i=T(\zeta_{g+2}^{i+2}t)$. 
     Lastly, let $D_t$ be the normalization of $H_t \times_{\Po} E$. Then $f_{D_t}\leq \frac{g}{2}-1+f_E$; in particular, $f_{D_t}\leq \frac{g}{2}$.
\end{proposition}
\begin{proof}
Observe that $D_t$ has two hyperelliptic quotients: 
\[H_t: z^2=\prod_{k=1}^{g-1}(x-x_k)\hspace{3mm} \text{ and } \hspace{3mm} H'_t: w^2=x(x-1)(x-\lambda)\prod_{k=1}^{g-1}(x-x_k).\] 
Furthermore, we have $f_{\Jac(D_t)_{\textup{new}}}=f_{H_t}+f_{H_t'}$. Since $g+2$ is even, the branching points of $H_t'$ are exactly $0,1,\lambda, x_1, \cdots, x_{g-1}$ and it is not branched at~$\infty$. By the change of variables $x'=T^{-1}x$, we see that $H_t'$ is isomorphic to the curve given by $w^2=x'^{g+2}-t^{g+2}$. By Lemma \ref{lem: p rank 0 construction}, we see that when $p \equiv -1 \pmod{g+2}$, then $f_{H_t'}=0$. Hence, $f_{\Jac(D_t)_{\textup{new}}} \leq f_{H_t} \leq g_{H_t} = \lceil\frac{g-1}{2}\rceil-1$. We conclude $f_{D_t} \leq \frac{g}{2}-1+f_E$.  
\end{proof}

When $g$ is odd, notice that the upper bound given by Proposition~\ref{prop: general $g$} is $\frac{g+1}{2}+1$. In the following proposition, we construct a family $\D_t$ whose $p$-rank upper bound becomes $\frac{g+1}{2}$ given that $g$ is odd. 

\begin{proposition}\label{prop: even n}
     Let $g>3$ be an odd number and $p \equiv -1 \pmod{g+1}$. Fix the elliptic curve $E: y^2=x(x-1)(x-\lambda)$ over $k$ and a primitive $(g+1)^{th}$ root of unity $\zeta_{g+1}$. For $t \in k^*$, let~$H_t$ denote the hyperelliptic curve $H_t: z^2=x(x-1)\prod_{k=1}^{g-1}(x-x_k)$, where $x_1, \dots, x_g \in \Po(k)$ are chosen such that, if $T$ is one of the linear fractional transformations that sends $(t, \zeta_{g+1}t)$ to $(0,1)$, then~$x_i=T(\zeta_{g+1}^{i+1}t)$. 
     Lastly, let $D_t$ be the normalization of $H_t \times_{\Po} E$. Then $f_{D_t}\leq \frac{g-1}{2}+f_E \leq \frac{g+1}{2}$.
\end{proposition}
\begin{proof}
Observe that $D_t$ has two hyperelliptic quotients: 
\[H_t: z^2=x(x-1)\prod_{k=1}^{g-1}(x-x_k)\hspace{3mm} \text{ and } \hspace{3mm} H'_t: w^2=(x-\lambda)\prod_{k=1}^{g-1}(x-x_k).\] 
Furthermore, we have $f_{\Jac(D_t)_{\textup{new}}}=f_{H_t}+f_{H_t'}$. Since $g+1$ is even, the branching points of $H_t$ are exactly $0,1, x_1, \cdots, x_{g-1}$ and it is not branched at~$\infty$. By the change of variable $x'=T^{-1}x$, we see that $H_t$ is isomorphic to the curve given by $w^2=x'^{g+1}-t^{g+1}$. By Lemma \ref{lem: p rank 0 construction}, we see that $p \equiv -1 \pmod{g+1}$ implies $f_{H_t}=0$. Hence, we conclude $f_{\Jac(D_t)_{\textup{new}}} \leq f_{H_t'} \leq g_{H_t'} = \lceil\frac{g}{2}\rceil-1$. Hence, $f_{D_t} \leq \frac{g+1}{2}-1+f_E$.  
\end{proof}

\appendix

\section{The geometry of $\overline{M}_{E,g}$}\label{appendix:geometry_meg}

In this appendix, we cover some of the more technical details of the geometry of $\overline{M}_{E,g}$ and its marked versions $\overline{M}_{E,g;n}$, including its construction as a stack of twisted stable maps and the proof of key properties like smoothness and properness.

\subsection{Stable maps to Deligne-Mumford stacks}\label{section:stable_maps_dm}

We will define $M_{E,g;n}$, which parametrizes genus-$g$ double covers $\pi: D \to E$ with $n$ marked points of $E$ above which $\pi$ is unramified. We will define $\mb_{E,g;n}$ as a closed substack of the moduli stack of (twisted) stable maps to $E \times B\mathbb{Z}/2\mathbb{Z}$ and then define $M_{E,g;n} \subset \mb_{E,g;n}$ as the open substack where the source curve is smooth. Before defining our specific moduli space, we take some time to recall the definition and basic properties of moduli stacks of stable maps to Deligne-Mumford stacks, as studied in \cite{abramovich2002compactifying} and \cite{abramovich2003twisted}. Recall that a Deligne-Mumford stack $\mathcal{X}$ over a base $S$ is \emph{tame} if the orders of all automorphism groups of $\mathcal{X}$ are invertible in $S$.

\begin{definition}[Variant of {\cite[Definition 4.1.2]{abramovich2002compactifying}}]\label{definition:twisted_curve}
A \emph{twisted nodal $n$-pointed curve over a scheme $S$} is a collection of data $(\mathcal{C} \to C \to S,\sigma_i)$, where \begin{enumerate}[(1)]
\item $\mathcal{C}$ is a tame Deligne-Mumford stack, proper over $S$, which is \'{e}tale locally a nodal curve over $S$;

\item For $1 \le i \le n$, $\sigma_i$ is a section $S \to C$ with image in the smooth locus of $C$;

\item The preimages of $\sigma_i(S)$ under the map $\mathcal{C} \to C$ are \'{e}tale gerbes over $S$: More precisely, the $\sigma_i(S)$ are locally isomorphic to $S \times B\mathbb{Z}/r\mathbb{Z}$ for some $r$; 

\item The morphism $\mathcal{C} \to C$ exhibits $C$ as the coarse moduli scheme of $\mathcal{C}$;

\item Outside the images of the $\sigma_i$, the map $\mathcal{C} \to C$ is an isomorphism over the smooth locus;

\item Above the nodes, $\mathcal{C}$ is \emph{balanced} in the following sense: \'{E}tale-locally above the nodes of $C$, the morphism $\mathcal{C} \to S$ looks like $$\left[\frac{\Spec (A[x,y]/(xy - t))}{\mu_r}\right] \to \Spec (A),$$ where $t$ is an element of $A$ and $\zeta \in \mu_r$ acts on $A[x,y]/(xy - t)$ by $(x,y) \mapsto (\zeta x,\zeta^{-1}y)$.
\end{enumerate}
\end{definition}

\begin{remark}\label{remark:orbi_balanced}
This differs from \cite[Definition 4.1.2]{abramovich2002compactifying} slightly in that we require the stacky nodes to be balanced (6). Because we are trying to compactify moduli spaces $M_{E,g;n}$ of smooth double covers, we are only interested in admissible covers that can deform to smooth covers, and admissible covers of balanced twisted curves are precisely the smoothable covers. Considering only balanced twisted curves is not an issue for representability of our moduli stacks because having balanced nodes is an open and closed condition.
\end{remark}

By \cite[Proposition 4.2.2]{abramovich2002compactifying}, the 2-category of twisted nodal $n$-pointed curves is equivalent to a 1-category, so it makes sense to talk about morphisms of twisted curves without having to worry about 2-morphisms.

Now to define twisted stable maps, let $\mathbb{S}$ be a fixed Noetherian scheme, and let $\mathcal{X}/\mathbb{S}$ be a proper tame Deligne-Mumford stack admitting a projective coarse moduli scheme $X$. Fix an ample invertible sheaf on $X$, so that it makes sense to talk about the degree of a map from a curve to $X$. 

\begin{definition}[{\cite[Definition 4.3.1]{abramovich2002compactifying}}]
A \emph{twisted stable $n$-pointed map of genus $g$ and degree $d$ over $S$} is a collection of data $(\mathcal{C} \to S,\sigma_i,f: \mathcal{C} \to \mathcal{X})$ such that
\begin{enumerate}[(1)]
\item $(\mathcal{C} \to C \to S,\sigma_i)$ is a genus-$g$ twisted nodal $n$-pointed curve;

\item $f$ is a representable morphism;

\item The map on coarse spaces $(C \to S,\sigma_i,f: C \to X)$ is a stable $n$-pointed map of degree $d$.
\end{enumerate}
\end{definition}

Again, by \cite[Proposition 4.2.2]{abramovich2002compactifying}, twisted stable maps form a 1-category, so we don't need to worry about 2-morphisms. 
This means that we can define the stack $\mathcal{K}_{g,n}^{\text{bal}}(\mathcal{X},d)$ over $\mathbb{S}$ of twisted stable $n$-pointed maps of genus $g$ and degree $d$, as in \cite[\S3.8]{abramovich2002compactifying}.

\begin{theorem}[{\cite[Theorem 1.4.1]{abramovich2002compactifying}}]\label{theorem:kgn_exists}
$\mathcal{K}_{g,n}^{\text{bal}}(\mathcal{X},d)$ is a proper algebraic stack. Furthermore, the map $\mathcal{K}_{g,n}^{\text{bal}}(\mathcal{X},d) \to \mathcal{K}_{g,n}(X,d)$ is of Deligne-Mumford type.
\end{theorem}

\subsection{The definition of $\overline{M}_{E,g}$ and $\mb_{E,g;n}$}\label{section:definition_meg}

In this section, we define the Deligne-Mumford stack $\mb_{E,g;n}$ as certain substack of the stack $\mathcal{K}_{1,2g-2+n}^{\text{bal}}(E \times B\Z/2\Z,1)$. In the rest of the paper, we require $g \geq 1, n \geq 0$, and $g+n \geq 2$ for $\mb_{E,g;n}$, so that the object in the moduli problem has finite automorphism group. 

\begin{definition}\label{def: mbEgn}
For $p \neq 2$, let $k$ be an algebraically closed field of characteristic $p$, and let $\mathbb{S}=\Spec(k)$. We define the following Deligne-Mumford stacks:
\begin{itemize}
    \item We define $\mt_{E,g;n}/\mathbb{S}$ as the open and closed substack of $\mathcal{K}_{1,2g-2+n}^{\text{bal}}(E \times B\Z/2\Z,1)$ parametrizing genus 1 stable maps
    \[(\mathcal{C} \to S, (\sigma_1, \dots, \sigma_{2g-2}, \tau_1, \dots, \tau_n), f: \mathcal{C} \to E \times B\Z/2\Z)\]
    such that the automorphism groups of points of $\mathcal{C}$ above $\sigma_i$ have order $2$, and the automorphism groups of points on $\mathcal{C}$ above the $\tau_j$ have order $1$. The number $g$ is the genus of the double cover $D \to \mathcal{C}$ classified by $\mathcal{C} \to B\Z/2\Z$.
    
    \item We define $\mb_{E,g;n} \coloneqq [\mt_{E,g;n}/S_{2g-2}\times S_{n}]$, where $S_{2g-2}$ acts by permuting the first $2g-2$ marked points (the branching locus) and $S_n$ permutes the $n$ unbranched marked points. 
    
    \item We further define $M_{E,g;n} \subset \mb_{E,g;n}$ as the open substack consisting of stable maps with smooth source curve.

    \item When there is no unbranched marked point (that is when $n=0$), we denote $\mt_{E,g}, \mb_{E,g}$ and $M_{E,g}$ for the corresponding Deligne-Mumford stack.
\end{itemize}
\end{definition}

\begin{remark}\label{rem: describing the $k$ point of $M_{E,g,n}$}
To unravel the definition of $\mb_{E,g;n}$, we describe its $k$-points. To this end, suppose $(\mathcal{C} \to \Spec(k),\Sigma_{\br},\Sigma_{\ub};f: \mathcal{C} \to E \times B\mathbb{Z}/2\Z)$ is a $k$-point of $\mb_{E,g;n}$. To describe the map $f: \mathcal{C} \to E \times B\mathbb{Z}/2\Z$, we separately describe the content of the maps $\mathcal{C} \to E$ and $\mathcal{C} \to B\mathbb{Z}/2\Z$:
\begin{itemize}
\item The map $\mathcal{C} \to E$: Describing a map $\mathcal{C} \to E$ is equivalent to describing a map of coarse spaces $C \to E$. This map $C \to E$ has degree 1, meaning that one component of the normalization $\widetilde{C}$ of $C$ maps isomorphically onto $E$, while all other components map to points on $E$. Because $C$ has genus 1, $C$ must look like a genus 1 curve identified with $E$ by the stable map with genus 0 bubbles attached to it. Because $f: C \to E$ is stable, all the genus 0 components of $C$ must have at least 3 special points.

\item The map $\mathcal{C} \to B\Z/2\Z$: By \cite[Theorem 4.3.2]{abramovich2003twisted}, this is equivalent to the data of an admissible $\mathbb{Z}/2\Z$-cover $D \to C$ in the sense of \cite[Definition 4.3.1]{abramovich2003twisted}. For the reader unfamiliar with admissible covers, an admissible $\mathbb{Z}/2\Z$-cover $D \to C$ is essentially a $\mathbb{Z}/2\Z$-cover where we only allow branching above the marked points and nodes. Because we require the stabilizers of $\mathcal{C}$ at $\Sigma_{\br}$ to be nontrivial and the map $\mathcal{C} \to B\mathbb{Z}/2\Z$ to be representable, the $\mathbb{Z}/2\Z$-cover $D \to C$ must be branched over $\Sigma_{\br}$. Similarly, the stabilizers of $\mathcal{C}$ at $\Sigma_{\ub}$ are trivial, so it is unbranched over $\Sigma_{\ub}$.

If $g\geq 2$, then the fact that there are $2g - 2$ branch points ensures that the genus of $D$ is $g$ by the Riemann-Hurwitz formula. We discuss the $g=1$ case in Remark \ref{rmk: g eq one}.
\end{itemize}

In summary, a $k$-point $(\mathcal{C} \to \Spec (k),\Sigma_{\br},\Sigma_{\ub},f: \mathcal{C} \to E \times B\mathbb{Z}/2\Z)$ of $\overline{M}_{E,g;n}$ is the same data as a nodal curve $C$ with one component identified with $E$ and all other components genus 0 bubbles, along with an admissible $\mathbb{Z}/2\Z$-cover $\pi: D \to C$ branched above $\Sigma_{\br}$ and possibly the nodes. The $k$-point lies in $M_{E,g;n}$ if $C = E$. Note that because $g+n \ge 2$, there must be at least one marked point, and $(C,\Sigma_{\br} \cup \Sigma_{\ub})$ must actually be stable.
This implies that $(D,\Sigma_{\br}^{\prime}\cup \Sigma_{\ub}^{\prime})$ is also stable, where $\Sigma_{\br}^{\prime}$ is the divisor of ramification points over the marked points $\Sigma_{\br}$ (i.e. $\Sigma_{\br}^{\prime} = \pi^{-1}(\Sigma_{\br})_{\mathrm{red}}$) and $\Sigma_{\ub}^{\prime} = \pi^{-1}(\Sigma_{\ub})_{\mathrm{red}}$.
\end{remark}

\begin{remark}\label{rmk: s-vz comparison}
Instead of constructing $\Bb_{E,g}$ as the image of the forgetful map $\mb_{E,g} \to \overline{\mathcal{M}}_g$, one could alternatively use the stack $\overline{\mathcal{H}}_{g,G,\xi}$ of admissible double covers of genus 1 curves, as defined in \cite[\S 1.3]{schmittvanzelm}, with Galois group $G = \mathbb{Z}/2\mathbb{Z}$ and ramification $\xi = (1,\ldots,1)$. In Section~\ref{Sec: Bielleptic locus}, we denote $\overline{\mathcal{B}}_g^{\rm adm} = \overline{\mathcal{H}}_{g,\mathbb{Z}/2\mathbb{Z},\xi}$.
Recall from Section~\ref{Sec: Bielleptic locus} that $\overline{\mathcal{B}}_g^{\rm adm}$ admits a map $\Psi: \overline{\mathcal{B}}_g^{\rm adm} \to \overline{\mathcal{M}}_g$, extracting the source curve, and a map $\Phi:\overline{\mathcal{B}}_g^{\rm adm} \to \overline{\mathcal{M}}_{1,1}$, extracting the target curve with one marked point. After fixing a point $[E]=(E,O_E) \in \overline{\mathcal{M}}_{1,1}$, we may alternatively define $\overline{\mathcal{B}}_{E,g}=\Psi(\Phi^{-1}([E]))$. The fact that this definition is equivalent to our definition follows from \cite[Theorem 3.7]{schmittvanzelm} and Proposition~\ref{forgetful map is quasi-finite}.
The reason we instead construct $\mb_{E,g}$ from scratch and realize $\Bb_{E,g}$ as the image of $\mb_{E,g}$ under the forgetful map is that our construction allows us to prove that $\Bb_{E,g}$ is irreducible. 

\end{remark}

\begin{remark}\label{rmk: g eq one}
    
    We take a moment to describe the points of $\overline{M}_{E,1}$, which correspond to \'{e}tale double covers of $E$. There are $4$ \'{e}tale double covers. Suppose $E[2](k)=\{O, P_1, P_2, P_3\}$, then the \'{e}tale covers of $E$ are of the following two types:
    \begin{enumerate}
        \item $f_0: E \sqcup E \to E$ is the disconnected $\Z/2\Z$-cover;
        \item $f_i: E/\langle P_i \rangle \to E$ for $1 \leq i \leq 3$, obtained as a factor of the multiplication by $2$ map on the elliptic curve $E$.
        \end{enumerate}
        Hence, $M_{E,1} \cong \mb_{E,1} \cong \mt_{E,1}$, and the coarse moduli space consists of four points.
\end{remark}
\begin{remark}\label{remark:arb_hurwitz}
A variation of the above construction where we consider stable maps to $C \times BG$ can be used to construct spaces of admissible $G$-covers of a fixed base curve $C$ for arbitrary $G$ and $C$, over a base scheme $\mathcal{S}$ such that $\#G$ is coprime to the characteristic of $\OO_S$.
\end{remark}

\subsection{Deformation theory of $\overline{M}_{E,g;n}$}\label{section:def_thy_meg}
In this section, we show that $\overline{M}_{E,g;n}$ is Deligne-Mumford and smooth and calculate its dimension.

\begin{proposition}\label{proposition:meg_proper_dm}
$\overline{M}_{E,g;n}$ is a proper Deligne-Mumford stack.
\end{proposition}

\begin{proof}
It suffices to show that $\mt_{E,g;n}$ is a proper Deligne-Mumford stack. By Theorem \ref{theorem:kgn_exists}, $\mt_{E,g;n}$ is a proper algebraic stack. Moreover, by Theorem \ref{theorem:kgn_exists}, to show that $\mt_{E,g;n}$ is Deligne-Mumford, it suffices to show that $\mathcal{K}_{1,2g - 2+n}^{\text{bal}}(E,1)$ is Deligne-Mumford.

Let $(C \to \Spec (k),\Sigma_{\br} \cup \Sigma_{\ub},f: C \to E)$ be a stable map corresponding to an object of $\mathcal{K}_{1,2g - 2+n}^{\text{bal}}(E,1)(k)$. We want to show that this object has no nontrivial infinitesimal automorphisms. If $C$ has more than one component, then it is stable, and $C$ automatically has no nontrivial infinitesimal automorphisms. Otherwise, $C \cong E$, and only the trivial infinitesimal automorphism can preserve the isomorphism $f$. Thus, $\mathcal{K}_{1,2g - 2+n}^{\text{bal}}(E,1)$ has no infinitesimal automorphisms and is Deligne-Mumford.
\end{proof}

We use the deformation-obstruction theory for stable maps in \cite[\S5.3]{abramovich2002compactifying} to prove that $\overline{M}_{E,g;n}$ is smooth and compute its dimension. Throughout the following proof, if $f: X \to Y$ is a map, we will use $\mathbb{L}_{X/Y}$ to denote the relative cotangent complex. If $Y = \Spec (k)$, we will just write $\mathbb{L}_X$.

\begin{theorem}\label{theorem:meg_smooth}
$\overline{M}_{E,g;n}$ is smooth of dimension $2g - 2+n$.
\end{theorem}

\begin{proof}
\noindent\textbf{Obstructions:} First, we show that $\overline{M}_{E,g;n}$ is smooth by showing the obstruction groups vanish. Let $A^{\prime} \to A$ be a small extension of local Artinian $k$-algebras with $I \coloneqq \mathrm{ker}(A^{\prime} \to A)$, and let $(\mathcal{C} \to \Spec (A),\Sigma_{\br} \cup \Sigma_{\ub},f: \mathcal{C} \to E \times B\mathbb{Z}/2\mathbb{Z})$ be an object of $\overline{M}_{E,g;n}(A)$. Moreover, define $\mathcal{C}_0 \coloneqq \mathcal{C} \times_{\Spec (A)} \Spec (k)$. By \cite[Lemma 5.3.3]{abramovich2002compactifying}, the obstruction to lifting this object to $A^{\prime}$ lies in $\Ext_{\mathcal{O}_C}^2(\mathbb{L}_{\mathcal{C}/E \times B\mathbb{Z}/2\mathbb{Z}},\mathcal{O}_{\mathcal{C}_0}) \otimes_k I \cong \Ext_{\mathcal{O}_{C_0}}^2(\mathbb{L}_{\mathcal{C}_0/E \times B\mathbb{Z}/2\mathbb{Z}},\mathcal{O}_{\mathcal{C}_0}) \otimes_k I$. For brevity, we will write $\Ext$ to mean $\Ext_{\mathcal{O}_{\mathcal{C}_0}}$ for the rest of this proof.

We claim that $\Ext^2(\mathbb{L}_{\mathcal{C}_0/E \times B\mathbb{Z}/2\Z},\mathcal{O}_{\mathcal{C}_0}) \cong 0$. Because the projection $E \times B\mathbb{Z}/2\Z \to E$ is \'{e}tale, we have $\mathbb{L}_{\mathcal{C}_0/E \times B\mathbb{Z}/2\Z} \cong \mathbb{L}_{\mathcal{C}_0/E}$. By an abuse of notation, we will use $f$ to denote the map $\mathcal{C}_0 \to E$. There is a distinguished triangle of $\mathcal{O}_{\mathcal{C}_0}$-modules: $$f^*\mathbb{L}_E \to \mathbb{L}_{\mathcal{C}_0} \to \mathbb{L}_{\mathcal{C}_0/E}.$$ Because $E$ is smooth, we have $\mathbb{L}_E \simeq \Omega_E$. Because $\mathcal{C}_0$ is nodal, a local computation shows that $\mathbb{L}_{\mathcal{C}_0}$ simplifies to $\Omega_{\mathcal{C}_0}$. Thus, we can simplify the first two terms to get $$f^*\Omega_E \to \Omega_{\mathcal{C}_0} \to \mathbb{L}_{\mathcal{C}_0/E}.$$ This distinguished triangle gives us an exact sequence $$\Ext^1(\Omega_{\mathcal{C}_0},\mathcal{O}_{\mathcal{C}_0}) \to \Ext^1(f^*\Omega_E,\mathcal{O}_{\mathcal{C}_0}) \to \Ext^2(\mathbb{L}_{\mathcal{C}_0/E},\mathcal{O}_{\mathcal{C}_0}) \to \Ext^2(\Omega_{\mathcal{C}_0},\mathcal{O}_{\mathcal{C}_0}),$$ so it suffices to show that $\Ext^2(\Omega_{\mathcal{C}_0},\mathcal{O}_{\mathcal{C}_0}) \cong 0$ and that $\Ext^1(\Omega_{\mathcal{C}_0},\mathcal{O}_{\mathcal{C}_0}) \to \Ext^1(f^*\Omega_E,\mathcal{O}_{\mathcal{C}_0})$ is surjective:
\begin{itemize}
\item $\Ext^2(\Omega_{\mathcal{C}_0},\mathcal{O}_{\mathcal{C}_0}) \cong 0$: This is shown in \cite[\S3.0.3]{abramovich2003twisted}. Alternatively, this follows from Serre duality \cite[Corollary 2.30]{nironi2008grothendieck} for the stacky curve $\mathcal{C}_0$ because the $\Ext^2$ is dual to an $H^{-1}$ group, which is necessarily zero.

\item $\Ext^1(\Omega_{\mathcal{C}_0},\mathcal{O}_{\mathcal{C}_0}) \to \Ext^1(f^*\Omega_E,\mathcal{O}_{\mathcal{C}_0})$ is surjective: Let $\omega_{\mathcal{C}_0}$ be the dualizing sheaf of $\mathcal{C}_0$ in the sense of \cite{nironi2008grothendieck}. By Serre duality \cite[Corollary 2.30]{nironi2008grothendieck}, the map of $\Ext$ groups is identified with the map $H^0(\mathcal{C}_0,\Omega_{\mathcal{C}_0} \otimes \omega_{\mathcal{C}_0})^{\vee} \to H^0(\mathcal{C}_0,f^*\Omega_E \otimes \omega_{\mathcal{C}_0})^{\vee}$. 

 To prove the surjectivity of the map in the statement, it suffices to prove that the map $$\psi: H^0(\mathcal{C}_0,f^*\Omega_E \otimes \omega_{\mathcal{C}_0}) \to H^0(\mathcal{C}_0,\Omega_{\mathcal{C}_0} \otimes \omega_{\mathcal{C}_0})$$ is injective. Let $s \in H^0(\mathcal{C}_0,f^*\Omega_E \otimes \omega_{\mathcal{C}_0})$ be a nonzero section. Because the dualizing sheaf is \'{e}tale local \cite[Lemma 2.12]{nironi2008grothendieck}, we can identify $\omega_{\mathcal{C}_0}$ with the sheaf of 1-forms with logarithmic singularities whose residues at both components of each node add to 0. In particular, any section of $f^*\Omega_E \otimes \omega_{\mathcal{C}_0}$ whose restriction to some component is nonzero must be nonzero on each component. This means that $s$ must be nonzero when restricted to the component of $\mathcal{C}_0$ whose coarse space is identified with $E$. On this component, the map $f^*\Omega_E \to \Omega_{\mathcal{C}_0}$ is generically an isomorphism and hence injective, which means that the image $\psi(s)$ is nonzero. We conclude that $\psi$ is injective and that $\Ext^1(\Omega_{\mathcal{C}_0},\mathcal{O}_{\mathcal{C}_0}) \to \Ext^1(f^*\Omega_E,\mathcal{O}_{\mathcal{C}_0})$ is surjective.
\end{itemize}
Having shown that $\Ext^2(\mathbb{L}_{\mathcal{C}_0/E \times B\mathbb{Z}/2\Z},\mathcal{O}_{\mathcal{C}_0}) \cong 0$, we deduce that $\overline{M}_{E,g;n}$ is smooth. \\

\noindent\textbf{Deformations:} 
We compute the dimension of $\mb_{E,g;n}$ by computing the dimension of tangent space at every point. Let $(\mathcal{C} \to \Spec (k), \Sigma, f: \mathcal{C} \to E \times B\Z/2\Z)$ be a object of $\mb_{E,g;n}(k)$. Then by \cite[Lemma 5.3.2]{abramovich2002compactifying}, the tangent space $T$ of $\mb_{E,g;n}$ at the point representing this object sits in the following exact sequence of vector spaces:
\[0 \to \Hom(\mathbb{L}_{\mathcal{C}/E \times B\Z/2}, \OO_\mathcal{C}) \to H^0(\mathcal{C}, \mathcal{N}_{\Sigma}) \to T \to \Ext^1(\mathbb{L}_{\mathcal{C}/E \times B\Z/2}, \OO_\mathcal{C})\to 0.\]
The exact sequence stated in \cite[Lemma 5.3.2]{abramovich2002compactifying} actually only gives us the above sequence from the second term $\Hom(\mathbb{L}_{\mathcal{C}/E \times B\Z/2}, \OO_\mathcal{C})$ onwards. However, as in \cite[\S3.0.4]{abramovich2003twisted}, we extend this exact sequence so that the first term is the infinitesimal automorphism group of the data $(\mathcal{C} \to \Spec (k), \Sigma, f: \mathcal{C} \to E \times B\Z/2\Z)$. By stability, this infinitesimal automorphism group vanishes, so we get the above exact sequence starting with 0.

First notice that $\mathcal{N}_{\Sigma}=\oplus_i \mathcal{N}_{\Sigma_i}$, where $\mathcal{N}_{\Sigma_i}$ is the normal sheaf at the $i^{th}$ marked point. By \cite[3.0.4]{abramovich2003twisted}, at each $\sigma_i$ where the marked is a branching point, then $H^0(\CC, \mathcal{N}_{\Sigma_i})=0$. At each $\tau_j$ where the marked point is not a branching point, we have $H^0(\CC, \mathcal{N}_{\Sigma_j}) \cong k$. Therefore, we obtain $H^0(\CC, \mathcal{N}_{\Sigma}) \cong k^{n}$. 

Next, since $E \times B\Z/2\Z \to E$ is \'{e}tale, we have $\LL_{\CC/E \times B\Z/2}\cong \LL_{\CC/E}$. We also have the distinguished triangle of $\mathcal{O}_{\mathcal{C}}$-modules:
\[f^*\Omega_E \to \Omega_{\CC} \to \LL_{\CC/E}.\]
This gives the long exact sequence
\begin{equation}
\begin{split}
&0 \to \Hom(\mathbb{L}_{\mathcal{C}/E},\mathcal{O}_{\mathcal{C}}) \to \Hom(\Omega_{\mathcal{C}},\mathcal{O}_{\mathcal{C}}) \to \Hom(f^*\Omega_E,\mathcal{O}_{\mathcal{C}}) \\ &\quad \to \Ext^1(\mathbb{L}_{\mathcal{C}/E},\mathcal{O}_{\mathcal{C}}) \to \Ext^1(\Omega_{\mathcal{C}},\mathcal{O}_{\mathcal{C}}) \to \Ext^1(f^*\Omega_E,\mathcal{O}_{\mathcal{C}}) \to 0.
\end{split}
\label{exseq:HomExt}
\end{equation}
As above, we use $\Ext^2(\LL_{\CC/E}, \OO_C)=0$. Hence, the dimension of $T$ is
\begin{equation*}
    \begin{split}
       \dim(T)&=n-\dim(\Hom(\LL_{\CC/E}, \OO_C)) +\dim(\Ext^1(\LL_{\CC/E}, \OO_C)) \\
       &=n-\dim(\Hom(\Omega_{\mathcal{C}},\mathcal{O}_{\mathcal{C}}))+\dim(\Ext^1(\Omega_{\mathcal{C}},\mathcal{O}_{\mathcal{C}}))+\dim(\Hom(f^*\Omega_E,\mathcal{O}_{\mathcal{C}}))-\dim(\Ext^1(f^*\Omega_E,\mathcal{O}_{\mathcal{C}})),
    \end{split}
\end{equation*}
by the exact sequence \eqref{exseq:HomExt}. Since $\Omega_E$ is trivial and $C$ has genus $1$, we observe
\begin{align*} 
\Hom(f^*\Omega_E, \OO_\CC) &\cong H^0(\CC, \OO_\CC)=H^0(C, \OO_C) \cong k\\
\Ext^1(f^*\Omega_E, \OO_\CC) &\cong H^1(\CC, \OO_\CC)=H^1(C, \OO_C) \cong k.
\end{align*}
By \cite[\S3.0.4]{abramovich2003twisted}, we have
$\Ext^i(\Omega_\CC, \OO_\CC)
=H^{1-i}(\CC, \Omega_{\CC}\otimes \omega_\CC)^{\smvee}
=H^{1-i}(C, \pi_*(\Omega_{\CC}\otimes \omega_\CC))^{\smvee}$, where $\pi: \mathcal{C} \to C$ is the coarse moduli space map.
Therefore, we have $$\dim(\Hom(\Omega_{\mathcal{C}},\mathcal{O}_{\mathcal{C}}))-\dim(\Ext^1(\Omega_{\mathcal{C}},\mathcal{O}_{\mathcal{C}}))=-\chi(\pi_*(\Omega_\CC \otimes \omega_\CC)).$$
Since all of our nodes are balanced, we have $\pi_*(\Omega_\CC \otimes \omega_\CC)=\Omega_C \otimes \omega_C(\Sigma_{\rm br})$ as in \cite[\S3.0.4]{abramovich2003twisted}, where $\Sigma_{\rm br}$ is summing over all the smooth branched markings $\sigma_i$.
Therefore, $$\chi(\Omega_C \otimes \omega_C(\Sigma_{\rm br}))=3g_{C}-3+(2g-2)=2g-2.$$ Hence, $\dim(T)=n+2g-2$, as desired. 
\end{proof}

\begin{remark}
    One might wonder to what extent it is possible to generalize this theorem, by replacing $E$ by an arbitrary curve or replacing $\Z/2\Z$ by an arbitrary cyclic group. In the more general setting, the deformation theory and dimension computation should still go through, giving us smooth compactifications of the relevant moduli spaces. However, for our inductive strategy of determining the dimensions of $p$-rank strata, it is harder to assess the boundary components.
\end{remark}

\begin{lemma}\label{Lem: Meg is irr}
For $g\geq 2$, $M_{E,g}$ is irreducible. Therefore, $\mb_{E,g}$ is irreducible.
\end{lemma}
\begin{proof}
To show that $M_{E,g}$ is irreducible, it suffices to exhibit a surjective map from an irreducible scheme $X$.
\begin{itemize}
    \item The construction of $X$: Let $\mathcal{L}$ be the universal bundle on $E \times \Pic^{g - 1}(E)$, and consider the projection $\pi: E \times \Pic^{g - 1}(E) \to \Pic^{g - 1}(E)$. Let $\mathcal{V} \coloneqq \pi_*\mathcal{L}^{\otimes 2}$. Because $2g - 2 > 0$ and $H^1(E,\mathcal{M}) \cong 0$ for any positive degree line bundle $\mathcal{M}$, by Serre duality, cohomology and base change tells us that $\mathcal{V}$ is a locally constant sheaf on $\Pic^{g - 1}(E)$. 
    
    By abuse of notation, let $\mathcal{V}$ also denote the space of global sections of $\mathcal{V}$ (as we will do throughout this proof). We let $X \subset \mathcal{V}$ be the open subset consisting of sections whose vanishing loci are reduced. In other words, $X$ consists of pairs $(\mathcal{M},r)$, where $\mathcal{M}$ is a line bundle on $E$ of degree $g - 1$ and the zeroes of $r \in H^0(E,\mathcal{M}^{\otimes 2})$ have multiplicity 1. Because $\Pic^{g - 1}(E)$ is irreducible, $X$ is irreducible.

\item The map $\phi: X \to M_{E,g}$: By definition,  the map $\phi$ can be defined by specifying a stable map $(\mathcal{C} \to X,\Sigma,f: \mathcal{C} \to E \times B\Z/2\Z)$ such that $f$ is degree 1, all the automorphism groups of $\mathcal{C}$ above $\Sigma$ have order 2, and $\mathcal{C} \to X$ is smooth. There exists a universal section $s: E \times X \to \left(\mathcal{L}|_{E \times X}\right)^{\otimes 2}$, where $\mathcal{L}|_{E \times X}$ is the pullback of $\mathcal{L}$ along the map $E \times X \to E \times \Pic^{g - 1}(E)$. We define $D \subset \mathcal{L}|_{E \times X}$ as the preimage of $s(E \times X) \subset \left(\mathcal{L}|_{E \times X}\right)^{\otimes 2}$ along the squaring map $\mathcal{L}|_{E \times X} \to \left(\mathcal{L}|_{E \times X}\right)^{\otimes 2}$. We then let $\mathcal{C} \coloneqq [D/(\mathbb{Z}/2\Z)]$, where $\mathbb{Z}/2\Z$ acts on $D \subset \mathcal{L}|_{E \times X}$ by multiplication by $\pm 1$. The projection $D \to E \times X$ exhibits $E \times X$ as the coarse quotient $C = D/(\mathbb{Z}/2\Z)$. Let $\Sigma \subset E \times X$ be the branch divisor of the double cover $D \to C$, which is the same as the vanishing locus of the universal section $s$. Because of the requirement that sections $r$ corresponding to points $(\mathcal{M},r) \in X$ have reduced vanishing loci, $\Sigma \to X$ is \'{e}tale. Checking that $(\mathcal{C} \to C \to X,\Sigma)$ is a smooth twisted $(g - 1)$-pointed curve can be done using the description of $\mathcal{C}$ as the quotient $[D/(\mathbb{Z}/2\Z)]$ and the fact that $\Sigma \to X$ is \'{e}tale, as these facts ensure that we have the local description we want for Definition \ref{definition:twisted_curve}.

We define $f: \mathcal{C} \to E \times B\mathbb{Z}/2\Z$ by letting $\mathcal{C} \to E$ be the projection $\mathcal{C} \to C = E \times X \to E$ and letting $\mathcal{C} \to B\mathbb{Z}/2\Z$ classify the \'{e}tale double cover $D \to \mathcal{C}$. The fact that $f$ is representable follows from the fact that $D$ is a scheme. The map on coarse spaces $C \to E$ is just the projection $E \times X \to E$, so it is a stable map of degree 1. The last thing we have to check is that the points in $\mathcal{C}$ above $\Sigma$ have stabilizer $\mathbb{Z}/2\Z$. This is true because the preimage of $\Sigma$ in $D$ is precisely the intersection of $D$ with the zero section of $\mathcal{L}|_{E \times X}$, which is stabilized by the $\mathbb{Z}/2\Z$-action on $D$. We thus have a twisted stable map $(\mathcal{C} \to X,\Sigma,f: \mathcal{C} \to E \times B\mathbb{Z}/2\Z) \in M_{E,g}$, defining a map $\phi: X \to M_{E,g}$.

\item $\phi$ is surjective:
This follows from the well-known fact that double covers of a curve can all be constructed as the square root of a section of a line bundle as above, which we will explain briefly below. In the following, we reset the notation we used earlier in the proof.
Let $(\mathcal{C} \to \Spec (k),\Sigma,f: \mathcal{C} \to E \times B\mathbb{Z}/2\Z)$ be any point of $M_{E,g}$. Then $\mathcal{C}$ has coarse space identified with $E$ by $f$, and $\mathcal{C} = [D/(\mathbb{Z}/2\Z)]$, where $D \to \mathcal{C} \to E$ is the double cover classified by the map $\mathcal{C} \to B\mathbb{Z}/2\Z$. We want to show that $h: D \to E$ can be produced using the square root construction above.
Let $\mathcal{M} \coloneqq (h_*\mathcal{O}_D/\mathcal{O}_E)^{\vee}$. Because 2 is invertible, the trace map provides a splitting $h_*\mathcal{O}_D \to \mathcal{O}_E$ of the exact sequence $0 \to \mathcal{O}_E \to h_*\mathcal{O}_D \to \mathcal{M}^{\vee} \to 0$, so that we get a section $\mathcal{M}^{\vee} \to h_*\mathcal{O}_D$. This extends to a surjection of $\mathcal{O}_E$-algebras $\Sym_{\mathcal{O}_E}^{\bullet}(\mathcal{M}^{\vee}) \to h_*\mathcal{O}_D$, which provides a closed embedding $D \xhookrightarrow{} \mathcal{M}$.
The ideal $\ker\left(\Sym_{\mathcal{O}_E}^{\bullet}(\mathcal{M}^{\vee}) \to h_*\mathcal{O}_D\right)$ is preserved by the $\mathbb{Z}/2\Z$-action acting by $-1$ on $\mathcal{M}^{\vee}$, so $D$ is preserved by the $\mathbb{Z}/2\Z$-action. This means that its image under the squaring map $\mathcal{M} \to \mathcal{M}^{\otimes 2}$ maps by degree 1 to $E$ and is thus isomorphic to $E$.
Thus, we have a section $r: E \to \mathcal{M}^{\otimes 2}$ whose square root gives us the double cover $D \to E$. Because $D$ is smooth, the vanishing locus of $r$ must be reduced, and we know that a point $(\mathcal{M},r) \in X$ maps to our point $(\mathcal{C} \to \Spec (k),\Sigma,f: \mathcal{C} \to E \times B\mathbb{Z}/2\Z)$ of $M_{E,g}$. We conclude that $h$ is surjective and that $M_{E,g}$ is irreducible.
\end{itemize}
\end{proof}

\begin{corollary}\label{corollary:megn_irred}
    $\mb_{E,g;n}$ is connected. Consequently, $\mb_{E,g;n}$ is irreducible. 
\end{corollary}
\begin{proof}

We show that $M_{E,g;n}$ is connected. Since it is smooth, this implies that it is irreducible.

We consider the forgetful morphism $\pi: M_{E,g;n} \to M_{E,g}$. The fiber is the configuration space, which is connected. If we let $\Delta$ denote the weak diagonal in $E^n$, and let $R$ be the branch locus of the fixed double cover, then the fiber over the double cover is $((E-R)^n -\Delta)/S_n$, which is connected.
Also, the image of $\pi$ is $M_{E,g}$, which is connected. Therefore, $M_{E,g;n}$ is connected, hence irreducible. By dimension, since $M_{E,g;n}$ is dense in $\mb_{E,g;n}$, it follows that $\mb_{E,g;n}$ is connected, hence irreducible. 
\end{proof}

\begin{proposition}\label{forgetful map is quasi-finite - appendix}
    We have the following morphism: 
    \[\mb_{E,g}^{\red} \xrightarrow{\varphi} \Mb_{g,2g-2} \xrightarrow{\pi} \Mb_{g}\]
    The first map $\varphi $ sends a cover $D \to C$ to its source curve $D$, with the $2g - 2$ ramification points as marked points.  The second map $\pi$ forgets the marked points and stabilizes the curve. Then $\pi \circ \varphi$ is quasi-finite on the smooth locus of $\mbr_{E,g}$ and generically finite on each boundary component of~$\mbr_{E,g}$.

\end{proposition}
\begin{proof}
Observe that $\pi \circ \varphi$ factors through $[\mb_{E,g}/\Aut(E)]$, and we already know that the map from $\mb_{E,g}^{\red}$ to $[\mb_{E,g}/\Aut(E)]$ is finite. Therefore, it is enough to show the statement for the map from $[\mb_{E,g}/\Aut(E)]$ to $\Mb_{g}$. Hence, we want to show that the forgetful map has finite fibers on the smooth locus and generically on each boundary component. Let $g \geq 2$.

\begin{enumerate}
    \item We first show quasifiniteness on the smooth locus. We know that $[M_{E,g}/\Aut(E)]$ parametrizes double covers $D\xrightarrow{\eta} E$ such that $D$ has genus $g$. Moreover, $D_1\xrightarrow{\eta_1} E$ and $D_2\xrightarrow{\eta_2} E$ correspond to the same point of $[M_{E,g}/\Aut(E)]$ if there are isomorphisms $f_1:D_1\to D_2$ and $f_2:E\to E$ such that $\eta_1=f_2\circ \eta_2 \circ f_1$.

Consider a curve $D$ in the image under $\pi \circ \varphi$ of the smooth locus of $\mbr_{E,g}$, and let $\eta: D \to E$ be the data of one of the preimages. We will show that it has only finitely many preimages in $[M_{E,g}/\Aut(E)]$. First, note that each double cover $\eta:D\to E$ gives us an element $\sigma\in \Aut(D)$ which is of order $2$. If we let $\eta^*$ and $\sigma^*$ to be the induced map on the function field $K(E)$ and $K(D)$ respectively, then we have that $\eta^*(K(E))=K(D)^{\sigma^*}$. Since $g\geq 2$, we know that $\Aut(D)$ is a finite group. Therefore, there are finitely many choices for $\sigma$ and hence finitely many choices for $\eta^*(K(E))$.

Now, we show that $\eta^*(K(E))$ determines the preimage. Suppose, we have two preimages $D \xrightarrow{\eta_0} E$ and $D\xrightarrow{\eta_1} E$ that satisfy $\eta_0^*(K(E))=\eta_1^*(K(E))$. Then there is $\tau\in \Aut(K(E))$, for which $\eta_1^* \circ \tau =\eta_0^*$ (this $\tau$ is obtained from the composition $K(E)\xrightarrow{\eta_0^*} \eta_0^*(K(D)) \xrightarrow{{\eta_1^{*}}^{-1}} K(E)$, which is well defined since $\eta_0^*(K(E))=\eta_1^*(K(E))$). Next, there is $f\in \Aut(E)$ for which $f^*=\tau$, so $f \circ \eta_1=\eta_0$.  This shows that $D \xrightarrow{\eta_0} E$ and $D\xrightarrow{\eta_1} E$ correspond to the same point of $[M_{E,g}/\Aut(E)]$. We conclude that the forgetful map from $[M_{E,g}/\Aut(E)]$ to $\Mb_{g}$ is quasi-finite on the smooth locus.

\item Next, consider a boundary component, either $\Du_{g_1,g_2-1}$ or $\Dr_{g_1,g_2}$. Curves in the generic part of any boundary component are the clutching of two smooth curves, except for one component of $\Du_{1,g-2}$, when it is the clutching of three smooth curves.  
Consider a curve $D \in \Mb_g$ coming from a generic point of a boundary component. Let $D'\xrightarrow{\eta_0} C_0 \xrightarrow{f_0} E$ be a preimage of $D$, where $C_0$ is the clutching of $E$ with $\Po$ (say $P_0$ of $E$ with $Q_0$ of $\Po$), and~$D$ is the stabilization of $D'$ after forgetting the ramification points. Let $D_{a}=\eta_0^{-1}(E)$ and $D_{b}=\eta_0^{-1}(\Po)$, so $g(D_{a})=g_1$ and $g(D_{b})=g_2$.
Then $D'$ is the clutching of $D_a$ and $D_b$, when $\eta_0^{-1}(P_0)$ is identified with $\eta_0^{-1}(Q_0)$. Notice that if $g_2\geq 1$, then $D=D'$.
On the other hand, if $g_2=0$, then it must be the case that $(D' \to C_0) \in \Du_{g-1,0}$. In this case, $D$ is $D_a$ with two points in $\eta_0^{-1}(P_0)$ identified. 

As in the smooth case, the double covers $\eta_0: D_a\to E$ and $\eta_0:D_b\to \Po$ determine elements of order two, $\sigma_a\in \Aut(D_a)$ and $\sigma_b\in \Aut(D_b)$ such that $\eta_0^*(K(E))=K(D_a)^{\sigma_a^*}$ and $\eta_0^*(K(\Po))=K(D_b)^{\sigma_b^*}$. For different values of $(g_1,g_2)$, we determine the possibilities for $D'$ and the map $D' \to E$. 
\begin{enumerate}
    \item If $g_1,g_2\geq 2$. Then, since a curve of genus $\geq 2$ has a finite automorphism group, we have only finitely many options for $\sigma_a$ and $\sigma_b$. We show that $\sigma_a, \sigma_b$ uniquely determine the map $D' \to C_0 \to E$. 
    
    Let $D'\xrightarrow{\eta_1} C_1\xrightarrow{f_1} E$ be another preimage of $D$ under $\pi \circ \varphi$. Suppose that we have $\eta_0^*(K(E))=\eta_1^*(K(E))$ and $\eta_0^*(K(\Po))=\eta_1^*(K(\Po))$. Then, there exist $h_a\in \Aut(E)$ and $h_b\in \Aut(\Po)$ such that $\eta_1^{*} \circ h_a^*= \eta_0^*$ and $\eta_1^{*} \circ h_b^*= \eta_0^*$, so $h_a\circ \eta_1=\eta_0 |_{D_a}$ and $h_b\circ \eta_1=\eta_0 |_{D_b}$.
    Recall that $D'$ is the clutching of $D_a$ and $D_b$, when $\eta_0^{-1}(P_0)\subseteq D_a$ is identified with $\eta_0^{-1}(Q_0)\subseteq D_b$. Since $D' \xrightarrow{\eta_1} C_1 \xrightarrow{f_1}E$ is another preimage of $D$,  we see that $D$ is also the clutching of $D_a$ and $D_b$, when $\eta_1^{-1}(P_1)$ is identified with $\eta_1^{-1}(Q_1)$. This means that $\eta_0^{-1}(P_0)=\eta_1^{-1}(P_1)$ and $\eta_0^{-1}(Q_0)=\eta_1^{-1}(Q_1)$. We see that $h_a(P_1)=P_0$ and $h_b(Q_1)=Q_0$. Thus the map $h=h_a \cup h_b$ is well defined from $C_1$ to $C_0$, it is an isomorphism and satisfies $\eta_0=h\circ \eta_1$. 
    
    Similarly, note that the map $f_0:C_0\to E$ is the identity on $E$ and maps $\Po$ to $P_0$, similarly the map $f_1:C_1\to E$ is the identity on $E$ and maps $\Po$ to $P_1$. Therefore, we have the following commutative diagram, in which $h$ and $h_a$ are isomorphisms:
    \[\begin{tikzcd}
    \arrow[d,"\eta_0"] D'  &\arrow[l,"id"] D'\arrow[d,"\eta_1"]   \\
    C_0\arrow[d,"f_0"] &\arrow[l,"h"] C_1\arrow[d,"f_1"]\\
    E &\arrow[l,"h_a"] E
    \end{tikzcd}\]
    It follows that $D'\xrightarrow{\eta_0} C_0\xrightarrow{f_0} E$ and $D'\xrightarrow{\eta_1} C_1\xrightarrow{f_1} E$ correspond to the same point of $[\mb_{E,g}/\Aut(E)]$. Since $\sigma_a^*$ and $\sigma_b^*$ determine $\eta_0^*(K(E))$ and $\eta_0^*(K(\Po))$, we see that $D$ has only finitely many preimages. 

    \item If $g_1\geq 2$ and $g_2=1$. Then $(D' \to C_0)$ must be in $\Dr_{g-1,1}$ or $\Du_{g-2,1}$. For now we consider both cases simultaneously and we distinguish between the different cases at the end. Note that there are only finitely many options for $\sigma_a$.
    Moreover, $D_b$ is an elliptic curve, so its $j$-invariant is the $j$-invariant of the cross-ratio of the four branch points of $\eta_0:D_b\to\Po$. Thus there are at most six choices for this cross-ratio. We show that $\sigma_a$ and the cross ratio uniquely determine $D' \to C_0 \to E$. 

    Let $D'\xrightarrow{\eta_1} C_1\xrightarrow{f_1} E$ be another preimage of $D$, and suppose that $\eta_0^*(K(E))=\eta_1^*(K(E))$ and $\eta_0:D_b\to\Po$ and $\eta_1:D_b\to\Po$ have the same cross-ratio of branch points. Then, similar to case~(a), there is $h_a\in \Aut(E)$ such that $\eta_1^{*} \circ h_a^*= \eta_0^*$, so $h_a\circ \eta_1=\eta_0 |_{D_a}$.
    
    Suppose $\eta_0:D_b\to \Po$ is branched at $\{Q_0,R_{0,1},R_{0,2},R_{0,3}\}$ and $\eta_1:D_b\to \Po$ is branched at $\{Q_1,R_{1,1},R_{1,2},R_{1,3}\}$. Since they have the same cross-ratio, there is a linear fractional transformation $h_b:\Po\to\Po$ that sends $\{Q_1,R_{1,1},R_{1,2},R_{1,3}\}$ to $\{Q_0,R_{0,1},R_{0,2},R_{0,3}\}$. Now, $\eta_0$ and $h_b\circ \eta_1:D_b\to \Po$ have the same branch points, so $\eta_0=h_b\circ \eta_1$. The rest of the argument is identical to case~(a), which shows that $D'\xrightarrow{\eta_0} C_0\xrightarrow{f_0} E$ and $D'\xrightarrow{\eta_1} C_1\xrightarrow{f_1} E$ correspond to the same point of $[\mb_{E,g}/\Aut(E)]$.

    Now, when $D$ is in $\Delta_{g-1,1}$, then the point on $D_b$ that is clutched with $D_a$ is one of the four branch points $\{Q,R_{1}, R_{2}, R_{3} \}$, so there are finitely many possibilities. On the other hand, when $D$ is in $\Du_{g-2,1}$, the clutching point of $D_b$ is any other point. The position of this point relative to $\{Q,R_{1},R_{2},R_{3}\}= D_b[2]$ introduces another dimension to the moduli problem. Two choices of a clutching point give the same point in $\Mb_g$ if and only if there is an automorphism of $D_b$ that exchanges the two clutching points and preserves $D_b[2]$, so again this gives a finite number of preimages in $[\overline{M}_{E,g}/\Aut(E)]$. 
    
    \item If $g_1\geq 2$ and $g_2= 0$. Then, $(D' \to C_0)$ is in $\Du_{g_1,0}$, and $D$ is $D_a$ with two points in $\eta_0^{-1}(P_0)$ identified with each other. Moreover, $g(D_b)=0$, so $D_b=\Po$. Again $\sigma_a$ has only finitely many choices.
    
    Suppose that $D'\xrightarrow{\eta_1} C_1\xrightarrow{f_1} E$ is another preimage of $D$ for which $\eta_0^*(K(E))=\eta_1^*(K(E))$. Similar to case~(a), there is $h_a\in \Aut(E)$ for which $h_a\circ \eta_1=\eta_0 |_{D_a}$. This forces $h_a(P_1)=P_0$. On the other hand, both $\eta_0 |_{D_b}$ and $\eta_1 |_{D_b}$ are double covers of $\Po$ ramified at two points, so there is $h_b\in \Aut(\Po)$ such that $\eta_0=h_b\circ \eta_1$ and $h_b(Q_1)=Q_0$. The rest of the proof is the same as in case~(a) to show that $D'\xrightarrow{\eta_0} C_0\xrightarrow{f_0} E$ and $D'\xrightarrow{\eta_1} C_1\xrightarrow{f_1} E$ correspond to the same point of $[\mb_{E,g}/\Aut(E)]$.

    \item If $g_1=1$ and $g_2\geq 2$. Then $\sigma_b$ has only finitely many choices and $\eta_0:D_a\to E$ has only $4$ choices (it is in $[\mb_{E,1}/\Aut(E)]$).
    Suppose that $D'\xrightarrow{\eta_1} C_1\xrightarrow{f_1} E$ is another preimage for which $\eta_0^*(K(\Po))=\eta_1^*(K(\Po))$ and $\eta_0,\eta_1:D_a\to E$ correspond to same point of $[\mb_{E,1}/\Aut(E)]$.
    Thus, there is $h_a\in \Aut(E)$ and $h_b\in \Aut(\Po)$ such that $\eta_0=h_a\circ \eta_1: D_a\to E$ and $\eta_0=h_b\circ \eta_1: D_b\to \Po$. The rest of the proof is identical to case~(a).

    \item If $g_1=1$ and $g_2=1$, so $D' \to E$ is in $\Dr_{1,1}$. 
    This case is identical to case~(b), except that $\eta_0:D_a\to E$ has only $4$ choices as an element of $[\mb_{E,1}/\Aut(E)]$.

    \item If $g_1=1$ and $g_2=0$, so $D' \to E$ is in $\Du_{1,0}$. This case is identical to case~(c), except that $\eta_0:D_a\to E$ has only $4$ choices as an element of $[\mb_{E,1}/\Aut(E)]$.
    
\end{enumerate}
\end{enumerate}

In all cases, we see that $D$ has only finitely many preimages in $[\mb_{E,g}/\Aut(E)]$. Thus, the map from $[\mb_{E,g}/\Aut(E)]$ to $\Mb_g$ is generically quasi-finite on each boundary divisor.
\end{proof}

\bibliographystyle{alpha}
\bibliography{biblio}

\end{document}